\documentclass[reqno]{amsart}
\usepackage[bookmarks=false]{hyperref} 
\usepackage{amssymb}
\usepackage{amsmath}
\usepackage{amsthm} 
\usepackage{mathtools}
\usepackage{color} 
\usepackage[svgnames]{xcolor}
\usepackage[utf8]{inputenc}
\usepackage[T1]{fontenc}
\usepackage{graphicx}
\usepackage{placeins}
\usepackage{vmargin}
\usepackage{setspace}

\hypersetup{
  colorlinks=true,
  linkcolor=Blue  ,          % color of internal links (change box color with linkbordercolor)
  citecolor=Red,        % color of links to bibliography
  filecolor=Magenta,      % color of file links
  urlcolor=Green,           % color of external links
pdfpagemode=UseOutlines,
pdftitle={},
pdfauthor={Tin Van Phan <van-tin.phan@univ-tlse3.fr>},
pdfsubject={On the Cauchy problem for
  a derivative nonlinear Schrodinger equation with nonvanishing boundary conditions},
pdfkeywords={Cauchy problem,derivative nonlinear Schr\"odinger equation,nonvanishing boundary conditions} 
}

\numberwithin{equation}{section}
\numberwithin{figure}{section}

\newtheorem{theorem}{Theorem}[section]
\newtheorem{proposition}[theorem]{Proposition}
\newtheorem{lemma}[theorem]{Lemma}

\newtheorem*{notation}{Notation}

\theoremstyle{definition}
\newtheorem{definition}[theorem]{Definition}

\theoremstyle{remark}
\newtheorem{remark}[theorem]{Remark}

\DeclarePairedDelimiter{\norm}{\lVert}{\rVert}

\newcommand{\eps}{\varepsilon}

\newcommand{\R}{\mathbb{R}}

\renewcommand{\leq}{\leqslant}
\renewcommand{\geq}{\geqslant}

\DeclareMathAlphabet{\mathpzc}{OT1}{pzc}{m}{it}
\renewcommand{\Re}{\mathcal R\!\mathpzc{e}}
\renewcommand{\Im}{\mathcal I\!\mathpzc{m}}

\begin{document}

\title[On the Cauchy problem for derivative NLS]{On the Cauchy problem for
  a derivative nonlinear Schr\"odinger equation with nonvanishing boundary conditions}

\author[Phan Van Tin]{Phan Van Tin}

\address[Phan Van Tin]{Institut de Math\'ematiques de Toulouse ; UMR5219,
  \newline\indent
  Universit\'e de Toulouse ; CNRS,
  \newline\indent
  UPS IMT, F-31062 Toulouse Cedex 9,
  \newline\indent
  France}
\email[Phan Van Tin]{van-tin.phan@univ-tlse3.fr}


\date{\today}

\begin{abstract} In this paper we consider the Schr\"odinger equation with nonlinear derivative term. Our goal is to initiate the study of this equation with non vanishing boundary conditions. We obtain the local well posedness for the Cauchy problem on Zhidkov spaces $X^k(\R)$ and in $\phi+H^k(\R)$. Moreover, we prove the existence of conservation laws by using localizing functions. Finally, we give explicit formulas for stationary solutions on Zhidkov spaces.
\end{abstract}

\maketitle
\tableofcontents

\section{Introduction}

We are interested in the Cauchy problem for the
following  derivative nonlinear Sch\"odinger equation  with nonvanishing boundary conditions:
\begin{equation}
\label{eq:1}
\begin{cases}
i\partial_t u+\partial^2 u = -i u^2\partial\overline{u} ,\\
u(0) = u_0,
\end{cases}
\end{equation}
where $u:\mathbb R_t\times \mathbb R_x\to \mathbb C$, $\partial=\partial_x$ denotes derivative in space and $\partial_t$ denotes derivative in time.

Our attention was drawn to this equation by the work of Hayashi and
Ozawa \cite{HaOz92} concerning the more general nonlinear Schr\"odinger equation
\begin{equation}
\label{eq:general}
\begin{cases}
i\partial_t u + \partial^2 u = i \lambda |u|^2\partial u +  i \mu u^2 \partial\overline{u} + f(u) ,\\
u(0) = u_0. 
\end{cases}
\end{equation}
When $\lambda = 0$, $\mu = -1$, $f \equiv 0$, then \eqref{eq:general}
reduces to \eqref{eq:1}.  This type of equation is usually refered to
as \emph{derivative nonlinear Schr\"odinger equations}. It may appear
in various areas of physics, e.g. in Plasma Physics for the
propagation of Alfv\'en waves \cite{Mj76,SuSu99}.

Under Dirichlet boundary conditions in space, the Cauchy problem for \eqref{eq:1} has been solved in \cite{HaOz92}: local well-posedness holds in $H^1(\R)$, i.e. for any $u_0\in H^1(\R)$ there exists a unique solution $u\in C(I,H^1(\R))$ of \eqref{eq:1} on a maximal interval of time $I$. Moreover, we have continuous dependence with respect to the initial data, blow-up at the ends of the time interval of existence $I$ if $I$ is bounded and conservation of energy, mass and momentum. 

The main difficulty is the appearance of the derivative term $-iu^2\overline{u_x}$. We cannot use the classical contraction method for this type of nonlinear Schr\"odinger equations. In \cite{HaOz92} Hayashi and Ozawa use the Gauge transform to establish the equivalence of the local  well-posedness between the equation \eqref{eq:general} and a system of equations without derivative terms. By studying the Cauchy problem for this system, they obtain the associated results for \eqref{eq:general}. In \cite{HaOz16}, Hayashi and Ozawa construct a sequence of solutions of approximated equations and prove that this sequence is converging to a solution of \eqref{eq:general}, obtaining this way the local well-posedness of \eqref{eq:general}. The approximation method has also been used by Tsutsumi and Fukuda in \cite{TsFu80,TsFu81}. The difference between \cite{HaOz16} and \cite{TsFu80,TsFu81} lies in the way of constructing the approximate equation. In \cite{HaOz16}, the authors use approximation on the non-linear term, whereas in \cite{TsFu80, TsFu81} the authors use approximation on the linear operator.

To our knowledge, the Cauchy  problem for \eqref{eq:1} has not been studied under non-zero boundary conditions, and our goal in this paper is to initiate this study. Note that non-zero boundary conditions on the whole space are much rarely considered in the literature around nonlinear dispersive equations than Dirichlet boundary conditions. In the case of the nonlinear Schr\"odinger equation with power-type nonlinearity, we refer to the works of G\'erard \cite{Ge06,Ge08} for local well-posedness in the energy space and to the works of Gallo \cite{Ga04} and Zhidkov \cite{Zh01} for local well-posedness in Zhidkov spaces (see Section \ref{sec:preliminaries} for the definition of Zhidkov spaces) and Gallo \cite{Ga08} for local well-posedness in $u_0+H^1(\R)$. In this paper, using the method of Hayashi and Ozawa as in \cite{HaOz92} on the Zhidkov-space  $X^k(\R)$, ($k \geq 4$) and in the space $\phi+H^k(\R)$ ($k=1,2$) for $\phi$ in a Zhidkov space, we obtain the existence, uniqueness and continuous dependence on the initial data of solutions of \eqref{eq:1} in these spaces. Using the transform 
\begin{equation}\label{eq:transform}
v=\partial u + \frac{i}{2} |u|^2u,
\end{equation}
we see that if $u$ is a solution of \eqref{eq:1} then $(u,v)$ is a solution of a system of two equations without derivative terms. It is easy to obtain the local wellposedness of this system on Zhidkov spaces. The main difficulty is how to obtain a solution of \eqref{eq:1} from a solution of the system. Actually, we must prove that the relation \eqref{eq:transform} is conserved in time. The main difference in our setting with the setting in \cite{HaOz94} is that we work on Zhidkov spaces instead of the space of localized functions $H^1(\R)$. Our first main result is the following.
\begin{theorem}\label{local well posed on Zhidkov space}
  Let $u_0 \in X^4(\R)$. Then there exists a unique maximal solution of \eqref{eq:1} $u \in C((-T_{min},T^{max}),X^4(\mathbb{R})) \cap C^1((-T_{min},T^{max}), X^2(\mathbb{R}))$. Moreover, $u$ satisfies the two following properties.
  \begin{itemize}
\item \emph{Blow-up alternative.} If $T^{max}$ (resp. $T_{min}$)$< +\infty$ then
  \[
    \lim_{t \to T^{max} \text{(resp. } -T_{min}\text{)}} \norm{u(t)}_{X^4} = \infty.
  \]
 \item   \emph{Continuity with respect to the initial data.}  If $u^n_0 \in X^4(\R)$ is such that $u^n_0 \to u_0$ in $X^4(\R)$ then for any subinterval $[T_1,T_2] \subset (-T_{min},T^{max})$ the associated solutions of equation \eqref{eq:1} $(u^n)$ verify 
    \[
      \lim_{n\to\infty}\norm{u^n-u}_{L^{\infty}([T_1,T_2],X^4)} =0.
    \]
  \end{itemize}
\end{theorem}

To obtain the local wellposedness on $\phi+H^k(\R)$ for $\phi$ in Zhidkov spaces $X^l(\R)$, we need to use the following transform
\begin{equation}\label{eq:transform2}
v = \partial u + \frac{i}{2}u(|u|^2-|\phi|^2)+\phi.
\end{equation}
We see that if $u$ is a solution of \eqref{eq:1} then $(u,v)$ is a solution of a system of two equations without the derivative terms. For technical reasons, we will need some regularity on $\phi$ and we take $l=4$. With a solution of the system in hand, we want to obtain a solution of \eqref{eq:1}. In practice, we need to prove that the relation \eqref{eq:transform2} is conserved in time. Our main second result is the following.

\begin{theorem}\label{thm local well posedness on phi plus H2 space}
  Let $\phi \in X^4(\R)$ and $u_0 \in \phi + H^2(\mathbb{R})$. Then the problem \eqref{eq:1} has a unique maximal solution $u \in C((-T_{min},T^{max}),\phi+H^2(\mathbb{R}))$ which is differentiable as a function of $C((-T_{min},T^{max}),\phi+ L^2(\R))$ and such that $u_t\in C((-T_{min},T^{max}), L^2(\R))$.

  Moreover $u$ satisfies the two following properties.\\
\textup{(1)} If $T^{max}$ (resp. $T_{min}$)$< \infty$ then 
\[
 \lim_{t \to T^{max} \text{(resp. } -T_{min}\text{)}}(\norm{u(t)-\phi}_{H^2(\mathbb{R})})=\infty.
\]
\textup{(2)} If $(u^n_0) \subset \phi+H^2(\mathbb{R})$ is such that $\norm{u^n_0-u_0}_{H^2} \to 0$ as $n \rightarrow \infty$ then for all $[T_1,T_2] \subset (-T_{min},T^{max})$ the associated solutions $(u^n)$ of \eqref{eq:1} satisfy
\[
\lim_{n \to \infty}\norm{u^n-u}_{L^{\infty}([T_1,T_2]),H^2} =0.
\]
\end{theorem}

In the less regular space $\phi+H^1(\R)$, we obtain the local well posedness under a smallness condition on the initial data. Our third main result is the following.

\begin{theorem}\label{local well posedness on phi plus H1 space}
Let $\phi \in X^4(\mathbb{R})$ such that $\norm{\partial\phi}_{H^2}$ is small enough, $u_0 \in \phi+H^1(\mathbb{R})$ such that $\norm{u_0-\phi}_{H^1(\mathbb{R})}$ is small enough. There exist $T>0$ and a unique solution $u$ of \eqref{eq:1} such that
\[
u-\phi \in C([-T,T],H^1(\mathbb{R})) \cap L^4([-T,T],W^{1,\infty}(\R)).
\]
\end{theorem}
In the proof of Theorem \ref{local well posedness on phi plus H1 space}, the main difference with the case $\phi+H^2(\R)$ is that we use Strichartz estimates to prove the contractivity of a map on $L^{\infty}([-T,T],L^2(\R)) \cap L^4([-T,T],L^{\infty}(\R))$. In the case of a general nonlinear term (as in \eqref{eq:general}), our method is not working. The main reason is that we do not have a proper transform to give a system without derivative terms. Moreover, our method is not working if the initial data lies on $X^1(\R)$. The main reason is that when we study the system of equations, we would have to study it on $L^{\infty}(\R)$, but we know that the Schr\"odinger group is not bounded from $L^{\infty}(\R)$ to $L^{\infty}(\R)$. Thus, the local wellposedness on less regular space is a difficult problem for nonlinear derivative Schr\"odinger equations.  
 
To prove the conservation laws of \eqref{eq:1}, we need to use a localizing function, which is necessary for integrals to be well defined. Indeed, to obtain the conservation of the energy, using \eqref{eq:1}, at least formally, we have
\[
\partial_t(|\partial u|^2)=\partial_x(F(u))+\partial_t(G(u)),
\]
for functions $F$ and $G$ which will be defined later. The important thing is that when $u$ is not in $H^1(\R)$, there are some terms in $G(u)$ which do not belong to $L^1(\R)$, hence, it is impossible to integrate the two sides as in the usual case. However, we can use a localizing function to deal with this problem. Similarly, we use the localizing function to prove the conservation of the mass and the momentum. The localizing function $\chi$ is defined as follows 
\begin{equation}
\chi \in C^1(\R), \quad
\text{supp}\chi \subset [-2,2],  \quad \text{ and }
\chi = 1 \, \text{on} \, [-1,1].\label{eq of chi}
\end{equation}
For all $R>0$, we define 
\begin{equation}\label{eq of chi R}
\chi_R(x) = \chi\left(\frac{x}{R}\right).
\end{equation}
Our fourth main result is the following.
\begin{theorem}
\label{thm conservation law special case}
Let $q_0 \in \R$ be a constant, $u_0 \in q_0 + H^2(\R)$ be such that $|u_0|^2-q_0^2 \in L^1(\R)$ and $u \in C((-T_{min},T^{max}),q_0+H^2(\R))$ be the associated solution of \eqref{eq:1} given by Theorem \ref{thm local well posedness on phi plus H2 space}. Then, we have 
\begin{align}
M(u)&:=\mathop{\lim}\limits_{R \rightarrow \infty}\int_{\R}(|u|^2-q_0^2)\chi_R \,dx =M(u_0), \quad \label{eq:conservation mass special}\\ 
E(u) &:=  \int_{\R} |\partial u|^2  \,dx +  \frac{1}{2}\Im\int_{\R} (|u|^2\overline{u}-q_0^3)\partial u\,  dx \nonumber\\
& + \frac{1}{6}\int_{\R}(|u|^2-|q_0|^2)^2(|u|^2+2|q_0|^2)  \,dx  =E(u_0), \quad \label{eq:conservation energy special} \\
P(u) &:= \frac{1}{2}\Im\int_{\R} (u-q_0)\partial\overline{u} \, dx-\int_{\R}\frac{1}{4}(|u|^2-|q_0|^2)^2  \,dx  = P(u_0) \quad \label{eq:conservation momentum special}.
\end{align}
for all $t \in (-T_{min},T_{max})$.
\end{theorem} 
\begin{remark}
\begin{itemize}
\item[(i)] When $q_0=0$, we recover the classical conservation of mass, energy and momentum as usually defined. 
\item[(ii)] Using the assumption $|u_0|^2-q_0^2 \in L^1(\R)$, we obtain 
\[
M(u_0)=\mathop{\lim}\limits_{R\rightarrow \infty}\int_{\R}(|u_0|^2-q_0^2)\chi_R \, dx=\int_{\R}(|u_0|^2-q_0^2) \, dx.
\]
Moreover, the existence of the limit $\mathop{\lim}\limits_{R\rightarrow \infty}\int_{\R}(|u|^2-q_0^2)\chi_R \, dx$ does not imply that $|u|^2-q_0^2 \in L^1(\R)$. It means that the property $|u|^2-q_0^2 \in L^1(\R)$ is not conserved in time.
\end{itemize}
\end{remark}
In the classical Schr\"odinger equation, there are special solutions which are called \emph{standing waves}. There are many works on standing waves (see e.g \cite{Stefan09}, \cite{Ca03} and the references therein). In \cite{Zh01}, Zhidkov shows that there are two types of bounded solitary waves possessing limits as $x\rightarrow \pm\infty$. These are monotone solutions and solutions which have precisely one extreme point. They are called \emph{kinks} and \emph{soliton-like solutions}, respectively. In \cite{Zh01}, Zhidkov studied the stability of kinks of classical Schr\"odinger equations. In \cite{BeGrSm14}, the authors have studied the stability of kinks in the energy space. To our knowledges, all these solitary waves are in Zhidkov spaces i.e the Zhidkov space is largest space we know to find special solutions. We want to investigate stationary solutions of \eqref{eq:1} in Zhidkov spaces. Our fifth main result is the following.

\begin{theorem}\label{thm kink soliton}
Let $\phi$ be a stationary solution of \eqref{eq:1} (see Definition \ref{definitionofstationarysolution}). Assume that $\phi$ is not a constant function and satisfies 
\[
\mathop{\inf}\limits_{x\in\R}|\phi(x)|:=m>0
\]
Then $\phi$ is of the form $e^{i\theta}\sqrt{k}$ where 
\begin{align*}
k(x)=2\sqrt{B}+\frac{-1}{\sqrt{\frac{5}{72B}}\cosh(2\sqrt{B}x)+\frac{5}{12\sqrt{B}}}, & \quad  \theta = \theta_0-\int_x^{\infty} \left(\frac{B}{k(y)}-\frac{k(y)}{4}\right)\, dy,
\end{align*}
for some constants $\theta_0 \in\R$, $B>0$. Moreover, if $\phi$ is a stationary solution of \eqref{eq:1} such that $\phi(\infty)=0$ then $\phi \equiv 0$ on $\R$.
\end{theorem}               

\begin{remark}
We have classified stationary solutions of \eqref{eq:1} for the functions which are vanishing at infinity, and for the functions which are not vanishing on $\overline{\R}$. One question still unanswered is the class of stationary solutions of \eqref{eq:1} vanishing at a point in $\R$.  
\end{remark}

This paper is organized as follows. In Section 2, we give the proof of local well posedness of solution of \eqref{eq:1} on Zhidkov spaces. In Section 3, we prove the local well posedness on $\phi+H^2(\R)$ and $\phi+H^1(\R)$, for $\phi \in X^4(\R)$ a given function. In Section 4, we give the proof of conservation laws when the initial data is in $q_0+H^2(\R)$, for a given constant $q_0 \in \R$. Finally, in Section 5, we have some results on stationary solutions of \eqref{eq:1} on Zhidkov spaces.  

\begin{notation}
In this paper, we will use in the following the notation $L$ for the linear part of the Schr\"odinger equation, that is
\[
L=i\partial_t+\partial^2.
 \] 
Moreover, $C$ denotes various positive constants and $C(R)$ denotes  constants depending on $R$. 
\end{notation}

\subsection*{Acknowledgement}
The author wishes to thank Prof.Stefan Le Coz for his guidance and encouragement.

\section{Local existence in Zhidkov spaces}
\label{sec local wellposedness on Zhidkov spaces}
In this section, we give the proof of Theorem \ref{local well posed on Zhidkov space}.
% We rewrite the equation \eqref{eq:1} in the following form
% \begin{equation}
% \label{eq:2}
% \begin{cases}
% Lu = -i u^2 \partial\overline{u} ,\\
% u(0)= u_0.
% \end{cases}
% \end{equation}
\subsection{Preliminaries on Zhidkov spaces}
\label{sec:preliminaries}

Before presenting our main results, we give some preliminaries. We start by recalling the definition of Zhidkov spaces, which were introduced by Peter Zhidkov in his pioneering works on Schr\"odinger equations with non-zero boundary conditions (see \cite{Zh01} and the references therein).

\begin{definition}
Let $k\in\mathbb N$, $k\geq 1$. The \emph{Zhidkov space} $X^k(\mathbb
R)$ is defined by
\[
X^k(\mathbb{R}) = \{ u \in L^{\infty}(\mathbb{R}) :\partial u \in
H^{k-1}(\mathbb{R})\}.
\]
It is a Banach space when endowed with the norm
\[
\norm{\cdot}_{X^k}=\norm{\cdot}_{L^\infty}+\sum_{\alpha=1}^k\norm{\partial^\alpha\cdot}_{L^2}.
\]
\end{definition}

It was proved by Gallo \cite[Theorem 3.1 and Theorem 3.2]{Ga04} that the
Schr\"odinger operator defines a group on Zhidkov spaces. More
precisely, we have the following result.

\begin{proposition}
\label{pro:group}
Let $k \geq 1$ and $u_0 \in X^k(\mathbb{R})$. For $t \in \mathbb{R}$ and $x \in \mathbb{R}$, the quantity
\begin{equation}
S(t)u_0(x):=\left\{
\begin{aligned}
e^{-i\pi /4} \pi^{-1/2} \lim_{ \varepsilon \to 0} \int_{\mathbb R} e^{(i-\varepsilon) z^2}u_0(x+2\sqrt{t} z) dz &\text{ if } t \geq 0 ,\\
e^{i\pi /4} \pi^{-1/2} \lim_{ \varepsilon \to 0} \int_{\mathbb R} e^{(-i-\varepsilon) z^2}u_0(x+2\sqrt{-t} z) dz  &\text{ if } t \leq 0 . 
\end{aligned}
\right.
\end{equation}
is well-defined and $S$ defines a strongly continuous group on $X^{k}(\mathbb R)$.
For all $u_0 \in X^k(R)$ and $t \in \R$ we have
\[
  \norm{S(t)u_0}_{X^k} \leq C(k)(1+ |t|^{1/4})\norm{u_0}_{X^k}.
  \]
The generator of the group $(S(t))|_{t\in\mathbb{R}}$ on $X^k(\mathbb{R})$ is $i\partial^2$ and its domain is $X^{k+2}(\mathbb{R})$. 
\end{proposition} 

\begin{remark}
Since, for all $\phi \in X^k(\R)$, we have $\phi+H^k(\R) \subset X^k(\R)$, the uniqueness of solution in $X^k(\R)$ implies the uniqueness of solution in $\phi+H^k(\R)$, and the existence of solution in $\phi+H^k(\R)$ implies the existence of solution in $X^k(\R)$.
\end{remark}

\subsection{From the equation to the system}
\label{sec:from-equation-system}
The equation \eqref{eq:1} contains a spatial derivative of $u$ in the nonlinear part, which makes it difficult to work with. In the following proposition, we indicate how to eliminate the derivative in the nonlinearity by introducing an auxiliary  function and converting the equation into a system.
\begin{proposition}
    \label{pro:from-eq-to-syst}
Let $k\geq 2$. Given $u\in X^k(\R)$, we define $v$ by
\begin{equation}\label{eq:v}
v=\partial u + \frac{i}{2} |u|^2u.
\end{equation}
Hence, $v \in X^{k-1}(\R)$. Furthermore, if $u$ satisfies the equation \eqref{eq:1}, then the couple $(u,v)$ verifies the system
\begin{equation}
  \label{eq:5}
  \left\{
  \begin{aligned}
    Lu&=P_1(u,v),\\
    Lv&=P_2(u,v),
  \end{aligned}
  \right.
\end{equation}
where $P_1$ and $P_2$ are given by
  \begin{equation}
    \label{eq:6}
\begin{aligned}
P_1(u,v) &= -iu^2\overline{v} + \frac{1}{2}|u|^4u ,\\
P_2(u,v) &= i\overline{u}v^2+\frac{3}{2}|u|^4v + u^2|u|^2\overline{v}.
\end{aligned}
\end{equation}
  \end{proposition}

  \begin{proof}
Let $u$ be a solution of \eqref{eq:1} and $v$ be defined by \eqref{eq:v}. 
Then we have
\begin{equation*}
Lu = -i u^2 \partial\overline{u} = -i u^2 \left(\overline{v}  +\frac{i}{2} (|u|^2 \overline{u}\right) 
 = -i u^2 \overline{v} + \frac{1}{2} |u|^4 u,
\end{equation*}
which gives us the first equation in \eqref{eq:5}.

On the other hand, since $L$ and $\partial$ commute and $u$ solves \eqref{eq:1}, we have
\begin{multline} 
Lv = \partial (Lu) + \frac{i}{2} L(|u|^2u)  
= \partial (-iu^2 \partial\overline{u}) + \frac{i}{2} L(|u|^2u) = -i (u^2 \partial^2\overline{u}+2u|\partial u|^2) + \frac{i}{2} L(|u|^2u) 
.
\label{eq:3}
\end{multline}
Using 
\begin{equation}
  \label{eq:L-on-product}
L(uv) = L(u)v+ u L(v) + 2 \partial u \partial v ,\quad
L(\overline{u}) = - \overline{Lu} + 2 \partial^2\overline{u},
\end{equation}
we have
\begin{multline}
  \label{eq:L-u-cube}
L(|u|^2u) = L(u^2 \overline{u}) 
= L(u^2) \overline{u} + u^2 L(\overline{u}) + 2 \partial(u^2)\partial\overline{u} \\
= \left(2L(u)u + 2 (\partial u)^2)\right) \overline{u} + u^2 (-\overline{Lu} + 2 \partial^2\overline{u}) + 4u|\partial u|^2 \\
= 2 L(u) |u|^2 + 2 \overline{u} (\partial u)^2 + 2 u^2 \partial^2\overline{u} - u^2 \overline{Lu} + 4u|\partial u|^2.
%= -2i u^2 \partial\overline{u} |u|^2 + 2\overline{u} (\partial u)^2 + 2 u^2 \partial^2 u - i |u|^4 \partial u + 4u |\partial u|^2. 
\end{multline}
We now recall that $u$ verifies \eqref{eq:1} to obtain
\begin{equation}
  \label{eq:i-L-u-cube}
\frac{i}{2} L(|u|^2u) = u^2 \partial\overline{u} |u|^2 + i \overline{u} (\partial u)^2 + i u^2 \partial^2 \overline{u} + \frac{1}{2} \partial u |u|^4 + 2i u |\partial u|^2.
\end{equation}
Subsituting in \eqref{eq:3}, we get
\begin{align*}
Lv &= -i(u^2 \partial^2\overline{u} + 2u|\partial u|^2) + u^2 \partial\overline{u}|u|^2 + i \overline{u}(\partial u)^2 + iu^2 \partial^2\overline{u} + \frac{1}{2} \partial u |u|^4 + 2iu|\partial u|^2 ,\\
&= u^2 \partial\overline{u}|u|^2 + i\overline{u}(\partial u)^2 + \frac{1}{2} \partial u |u|^4 .
\end{align*}
Observe here that the second order derivatives of $u$ have vanished and only first order derivatives remain. Therefore, using the expression of $v$ given in \eqref{eq:v} to subsitute $\partial u$, we obtain by direct calculations
\[
Lv = i\overline{u}v^2+\frac{3}{2}|u|^4v+u^2|u|^2\overline{v},
\]
which gives us the second equation in \eqref{eq:5}.
\end{proof}

\subsection{Resolution of the system}
\label{sec:resolution-system}

We now establish the local well-posedness of the system \eqref{eq:5} in Zhidkov spaces. 

\begin{proposition} \label{pro:10}
Let $k \geq 3$, and $(u_0 , v_0) \in X^k(\mathbb{R}) \times X^k(\mathbb{R})$. There exist $T_{min},T^{max} > 0$ and a unique maximal solution $(u,v)$ of system \eqref{eq:5} such that $(u,v) \in C((-T_{min} , T^{max}) , X^k(\R)) \cap C^1((-T_{min} , T^{max}) , X^{k-2}(\R))$. Furthermore the following properties are satisfied.
\begin{itemize}
\item \emph{Blow-up alternative.} If $T^{max}$ (resp. $T_{min}$)$< \infty$ then
  \[
    \lim_{t \to T^{max} \text{(resp. } T_{min}\text{)}} (\norm{u(t)}_{X^k} + \norm{v(t)}_{X^k}) = \infty.
  \]
\item \emph{Continuity with respect to the initial data.} If $(u^n_0,v^n_0) \in X^k\times X^k$  is such that
  \[
    \norm{u^n_0 - u_0}_{X^k} + \norm{v^n_0 - v_0}_{X^k} \to 0
  \]
  then  for any subinterval $[T_1,T_2] \subset (-T_{min},T^{max})$ the associated solution $(u^n,v^n)$ of \eqref{eq:5} satisfies
  \[
    \lim_{n \to \infty}\left(\norm{u^n-u}_{L^{\infty}([T_1,T_2], X^k)} + \norm{v^n-v}_{L^{\infty}([T_1,T_2], X^k)} \right) = 0.
  \]
\end{itemize}
\end{proposition}

\begin{proof}
  Consider the operator $A:D(A) \subset X^{k-2}(\R)\to X^{k-2}(\R)$ defined by  $A= i\partial^2$ with domain $D(A) = X^k(\R)$. From Proposition \ref{pro:group} we know that the operator $A$ is the generator of the Schr\"odinger group $S(t)$ on $X^{k-2}(\R)$.  From classical arguments (see \cite[Lemma 4.1.1 and Corollary 4.1.8]{CaHa98}) the couple $(u,v) \in C((-T_{min} , T^{max}) , X^k(\R)) \cap C^1((-T_{min} , T^{max}) , X^{k-2}(\R))$  solves  \eqref{eq:5} if and only if the couple $(u,v) \in C((-T_{min} , T^{max}) , X^k(\R))$ solves
\begin{equation} \begin{cases}
(u,v) = S(t)(u,v) - i \int_0^t S(t-s)P(u,v)(s) ds ,\\
u(0) = u_0 \in X^k(\R) , v(0) = v_0 \in X^k(\R) ,
\end{cases}
\end{equation}
where $S(t)(u,v):=(S(t)u,S(t)v)$, $P(u,v) = (P_1(u,v) , P_2(u,v))$ and $P_1$ and $P_2$ are defined in \eqref{eq:6}.
% such that 
% \begin{align*}
% P_1(u,v) &= -i u^2 (\overline{v} -1) + \frac{1}{2}u|u|^2(|u|^2-1) ,\\
% P_2(u,v) &= u^2(\overline{v}-1) \left(|u|^2 - \frac{1}{2}\right) + |u|^2(v-1)\left(\frac{3}{2}|u|^2-1\right) + i\overline{u}(v-1)^2.
% \end{align*}
Consider $P$ as a map from $X^k(\R) \times X^k(\R)$ into $X^k(\R) \times X^k(\R)$.
Since $P_1$ and $P_2$ are polynomial in $u$ and $v$, the map $P$ is Lipchitz continuous on bounded sets of $X^k(\R) \times X^k(\R)$.
% More precise, if $(u,v)$ and $ (x,y)$ are function in $X^k \times X^k$ such that $\norm{(u,v)}_{X^k \times X^k} , \norm{(x,y)}_{X^k \times X^k} < M$ for some $M>0$ then there exists the constant $C(M) > 0$ depend on $M$ such that
% \[
% \norm{P(u,v) - P(x,y)}_{X^k \times X^k} \leq C(M) \norm{(u,v) - (x,y)}_{X^k \times X^k}.
% \]
The result then follows from standard arguments (see \cite[Theorem 4.3.4 and Theorem 4.3.7]{CaHa98}).
\end{proof}

\subsection{Preservation of the differential identity}
\label{sec:pres-diff-ident}

The following proposition establishes the link from \eqref{eq:5} to \eqref{eq:1} by showing preservation along the time evolution of the differential identity 
\[
v_0 = \partial u_0 + \frac{i}{2} |u_0|^2 u_0. 
\]

\begin{proposition}\label{pro:15}
Let $u_0,v_0 \in X^3 (\mathbb{R})$ be such that
\[
v_0 = \partial u_0 + \frac{i}{2} u_0 |u_0|^2.  
\]
Then the associated solution $(u,v) \in C((-T_{min} ,T^{max}), X^3(\R) \times X^3(\R))$ obtained in Proposition \ref{pro:10} satisfies for all $t \in (-T_{min} , T^{max})$ the differential identity
\[
v = \partial u + \frac{i}{2} |u|^2 u .
\]
\end{proposition}

\begin{proof}
  Given $(u,v) \in C((-T_{min} ,T^{max}), X^3(\R) \times X^3(\R))$ the solution of \eqref{eq:5} obtained in Proposition \ref{pro:10}, we define 
  \[
w = \partial u + \frac{i}{2}|u|^2 u.
\]
Our goal will be to show that $w=v$. 
We first have
\begin{align*}
Lu &= -iu^2 \overline{v} + \frac{1}{2}|u|^4u \\
&= -iu^2(\overline{v} - \overline{w}) - iu^2 \overline{w}  + \frac{1}{2} |u|^4u \\
&= -iu^2(\overline{v} - \overline{w}) - iu^2 \partial\overline{u}.   
\end{align*}
Applying $L$ to $w$ and using \eqref{eq:L-u-cube} and the expression previously obtained for $Lu$, we get
\begin{align*}
L w &= \partial (Lu) + \frac{i}{2} L(|u|^2u)  \\
%&= \partial (Lu) + \frac{i}{2} \left( 2L(u)|u|^2 + 2\overline{u}(\partial u)^2 + u^2 L(\overline{u}) + 4u|\partial u|^2 \right) - \frac{i}{2} L(u) \\
&= \partial (Lu) + \frac{i}{2} \left(2 Lu |u|^2 + 2 \overline{u} (\partial u)^2 + 2 u^2 \partial^2\overline{u} - u^2 \overline{Lu} + 4u|\partial u|^2\right)  \\
&= \partial(-iu^2(\overline{v} - \overline{w}) - iu^2 \partial\overline{u}) + \frac{i}{2} \left(2(-iu^2 \partial\overline{u})|u|^2 + 2\overline{u}(\partial u)^2 - u^2 \overline{(-iu^2\partial\overline{u})}+ 2u^2 \partial^2\overline{u} + 4u|\partial u|^2 \right)  \\
&\quad + \frac{i}{2} \left[ 2 (-i u^2 (\overline{v} - \overline{w})) |u|^2 - u^2 \overline{(-iu^2(\overline{v} - \overline{w}))}\right]  \\
&= \left( -i\partial(u^2(\overline{v} - \overline{w})) + u^2|u|^2(\overline{v} - \overline{w}) + \frac{1}{2}|u|^4 (v-w) \right) \\
               &\quad + \left( -i\partial(u^2\partial\overline{u}) +
             u^2\partial\overline{u}|u|^2 + i\overline{u}(\partial u)^2 +\frac12|u|^4\partial u + iu^2 \partial^2\overline{u} + 2iu|\partial u|^2             
\right)   \\
&=: I_1 + I_2.
\end{align*}
As in the proof of Proposition \ref{pro:from-eq-to-syst}, we obtain
\[
I_2 = i\overline{u}w^2 + \frac{3}{2}|u|^4 w + |u|^2u^2\overline{w}. 
\]
Furthermore
\begin{align*}
I_1 &= \partial(-iu^2(\overline{v} - \overline{w})) + u^2|u|^2(\overline{v} - \overline{w}) + \frac{1}{2}|u|^4 (v-w)  \\
&= -iu^2\partial(\overline{v} - \overline{w}) - 2iu\partial u (\overline{v} - \overline{w}) + u^2|u|^2(\overline{v} - \overline{w}) + \frac{1}{2}|u|^4(v-w).
\end{align*}
It follows that
\begin{align}
Lw -Lv &= I_1 + (I_2 - Lv) \\
&=I_1 + i\overline{u}(w-v)(w+v)+\frac{3}{2}|u|^4(w-v)+|u|^2u^2(\overline{w}-\overline{v}) \\
&= (w-v) A_1 + (\overline{w} - \overline{v}) A_2 - iu^2 \partial(\overline{v} - \overline{w}),\label{NN}
\end{align}
where $A_1$ and $A_2$ are polynomials of degree at most $4$ in  $u$, $\partial u$, $v$, $\partial v$ and their complex conjugates.
Hence,
\begin{align} \label{eq:20}
(Lw - Lv)(\overline{w} -\overline{v}) &=|w -v|^2 A_1 + (\overline{w} - \overline{v})^2A_2 - iu^2 \frac{\partial(\overline{v} - \overline{w})^2}{2} := K,
\end{align}
where $K$ is a polynomial of degree at most $6$ in  $u$, $v$, $w$, $\partial u$, $\partial v$, $\partial w$ and their complex conjugates.
Remembering that $L = i\partial_t + \partial^2$, and taking imaginary part in the two sides of \eqref{eq:20} we obtain 
\begin{align}\label{eq:30}
\frac{1}{2} \partial_t |w -v|^2 + \Im(\partial\left((\partial w - \partial v)(\overline{w} - \overline{v})\right)) &= \Im(K).
\end{align}   
Let $\chi:\R\to\R$ be a cut-off function such that
\begin{equation*}
\chi \in  C^1(\R) , \quad
\operatorname{supp}(\chi) \subset [-2,2] ,\quad
\chi \equiv 1 \, \text{on} \, (-1,1) ,\quad
0 \leq \chi \leq 1 ,\quad
|\chi'(x)|^2 \lesssim \chi(x)\, \text{for all}\, x \in \mathbb{R}.
\end{equation*}
% \begin{equation*}
% \begin{cases}
% \chi \in C^1(\R) , \\
% supp \chi \subset B(0,2) ,\\
% \chi = 1 \, \text{on} \, (-1,1) ,\\
% 0 \leq \chi \leq 1 ,\\
% |\chi'(x)|^2 \lesssim \chi \forall \, x \in \mathbb{R}.
% \end{cases}
% \end{equation*}
For each $n \in \mathbb{N}$, define
\[
\chi_n(x) = \chi\left(\frac{x}{n}\right).
\]
Multiplying the two sides of \eqref{eq:30} with $\chi_n$ and integrating in space we obtain 
\begin{equation}
  \label{eq:40}
\frac{1}{2} \partial_t\norm{(w - v)\sqrt{\chi_n}}_{L^2}^2 + \int_{\mathbb{R}} \Im\left(\partial\left((\partial w - \partial v)(\overline{w} - \overline{v})\right)\right) \chi_n dx = \int_{\mathbb{R}}\Im(K) \chi_n  dx.
\end{equation}
For the right hand side, we have
\begin{equation*}
\int_{\mathbb{R}}\Im(K) \chi_n  dx = \Im\int_{\mathbb{R}} |w - v|^2A_1 \chi_n  dx + \Im\int_{\mathbb{R}} (\overline{w} - \overline{v})^2 A_2\chi_n  dx -Im \int_{\mathbb{R}}iu^2\frac{\partial((\overline{v} - \overline{w})^2)}{2} \chi_n  dx,
\end{equation*}
and therefore
\[
\left|\int_{\mathbb{R}}\Im(K) \chi_n  dx\right|\leq \norm{(w - v)\sqrt{\chi_n}}_{L^2}^2 \left( \norm{A_1}_{L^{\infty}} + \norm{A_2}_{L^{\infty}}\right) + \frac12 \left|\int_{\mathbb{R}} u^2 \partial((\overline{v} - \overline{w})^2) \chi_n dx\right|. 
  \]
We now fix some arbitrary interval $[-T_1,T_2]$ such that $0\in[-T_1,T_2] \subset (-T_{min},T^{max})$  in which we will be working from now on, and we set
\[
R = \norm{u}_{L^{\infty}([T_1,T_2],X^3)} + \norm{v}_{L^{\infty}([T_1,T_2],X^3)}.
\]
From the fact that $A_1$ and $A_2$ are  polynomials in  $u, \partial u$, $v$, $\partial v$ of degree at most  $4$, for all $t \in [T_1,T_2]$ we have
\[
\norm{A_1}_{L^{\infty}} + \norm{A_2}_{L^{\infty}} \leq C(R).
\]   
It follows that
\begin{align*}
\left|\int_{\mathbb{R}}\Im(K) \chi_n  dx\right| \leq \norm{(w -v)\sqrt{\chi_n}}_{L^2}^2 C(R) +\frac{1}{2} \left|\int_{\mathbb{R}}(\overline{v}-\overline{w})^2 \left(\partial (u^2)\chi_n + u^2 \partial\chi_n)  dx \right) \right|.
\end{align*}
By definition of $\chi$ we have
\begin{align*}
\left|\partial(u^2)\chi_n \right| &\leq C(R)\chi_n ,\\
\left|u^2 \partial\chi_n\right|&\leq |u^2| \frac{1}{n} \left|\chi'\left(\frac{\cdot}{n}\right)\right| \leq \frac{1}{n} C(R) \sqrt{\chi\left(\frac{\cdot}{n}\right)} \leq C(R) \frac{1}{n} \sqrt{\chi_n(.)}. 
\end{align*}
Hence,
\begin{align}
\left|\int_{\mathbb{R}}\Im(K) \chi_n  dx\right| &\leq\norm{(w -v)\sqrt{\chi_n}}_{L^2}^2 C(R) + \frac{C(R)}{n}\left|\int_{\R}(\overline{v}-\overline{w})^2 \sqrt{\chi_n}  dx \right| \nonumber\\
& \leq C(R) \norm{(w -v)\sqrt{\chi_n}}^2_{L^2} +\frac{C(R)^2}{n} \int_{\R}|v-w|\sqrt{\chi_n}  dx \nonumber\\
& \leq C(R)\norm{(w -v)\sqrt{\chi_n}}^2_{L^2}+ \frac{C(R)^2}{n} \int_{-2n}^{2n}|v-w|\sqrt{\chi_n}  dx \nonumber\\
& \leq C(R)\norm{(w -v)\sqrt{\chi_n}}^2_{L^2}+ \frac{C(R)^2}{n} \left(\int_{-2n}^{2n}(|v-w|\sqrt{\chi_n})^2  dx \right)^{\frac{1}{2}} \left(\int_{-2n}^{2n} dx\right)^{\frac{1}{2}} \nonumber\\
& \leq C(R)\norm{(w -v)\sqrt{\chi_n}}^2_{L^2}+ \frac{2C(R)^2}{\sqrt{n}}\norm{(w -v)\sqrt{\chi_n}}_{L^2}. \label{eq:50}
\end{align} 
In addition, we have
\begin{align}
\left|\int_{\mathbb{R}} \Im(\partial\left((\partial w -\partial v) (\overline{w} - \overline{v})\right)\chi_n)  dx\right| 
&= \left|\int_{\mathbb{R}} \Im(\left((\partial w -\partial v) (\overline{w} - \overline{v})\right)\chi_n')  dx\right| \nonumber \\
&= \left|\int_{\mathbb{R}} \Im\left((\partial w - \partial v) (\overline{w} - \overline{v})\frac{1}{n}\chi'\left(\frac{x}{n}\right)\right)dx\right| \nonumber\\
  & \leq \int_{\mathbb{R}} |\partial w - \partial v| |w - v| \frac{1}{n} \sqrt{\chi_n}  dx \nonumber\\
  & \leq \frac{1}{n}\norm{\partial w - \partial v}_{L^2} \norm{(w-v)\sqrt{\chi_n}}_{L^2} \nonumber\\
  &\leq \frac{C(R)}{n} \norm{(w-v)\sqrt{\chi_n}}_{L^2}.  \label{eq:60}
\end{align}
% \begin{align*}
% \left|\int_{\mathbb{R}} \Im(\partial\left((\partial w -\partial v) (\overline{w} - \overline{v})\right)\chi_n)  dx\right| 
% &= \left|\int_{\mathbb{R}}\Im\left((\partial w - \partial v)\left[\partial(\overline{w} - \overline{v})\chi_n + (\overline{w} - \overline{v})\partial\chi_n\right]\right)  dx\right| \\
% &= \left|\int_{\mathbb{R}} Im\left((\partial w - \partial v) (\overline{w} - \overline{v})\frac{1}{n}\chi'\left(\frac{x}{n}\right)\right)\right| \\
% & \leq \int_{\mathbb{R}} |\partial w - \partial v| |w - v| \frac{1}{n} \sqrt{\chi_n}  dx.
% \end{align*}
% where the final estimate given by
% \[
% \left|\chi'\left(\frac{x}{n}\right)\right| \leq \sqrt{\chi\left(\frac{x}{n}\right)} = \sqrt{\chi_n(x)}
% \]
% We obtain that
% \begin{align}
% \left|J\right| & \leq \int_{\mathbb{R}} |\partial w - \partial v||w-v| \sqrt{\chi_n} \frac{1}{n}  dx \\
% & \lesssim \frac{1}{n}\norm{\partial w - \partial v}_{L^2} \norm{(w-v)\sqrt{\chi_n}}_{L^2} \\
% & \lesssim \frac{1}{n}\norm{(w-v)\sqrt{\chi_n}}_{L^2}.  
% \label{eq:60}
% \end{align}
From \eqref{eq:40}, \eqref{eq:50}, \eqref{eq:60} we obtain that
\begin{align}
\partial_t\norm{(w - v)\sqrt{\chi_n}}^2_{L^2} &\leq C(R) \norm{(w -v)\sqrt{\chi_n}}^2_{L^2} + \frac{C(R)}{\sqrt{n}} \norm{(w -v)\sqrt{\chi_n}}_{L^2} \\
&\leq C(R)\norm{(w-v)\sqrt{\chi_n}}^2_{L^2} + \frac{C(R)}{\sqrt{n}} \label{eq:70}  
\end{align} 
where we have used the Cauchy inequality $|x| \leq \frac{|x|^2+1}{2}$. Define the function $g:[-T_1,T_2]$ by
\[
g= \norm{(w-v)\sqrt{\chi_n}}^2_{L^2}.
\]
Then by definition of $w$ we have $g(t=0) = 0$. Furthermore, from \eqref{eq:70} we have
\[
\partial_t g \leq C(R) g + \frac{C(R)}{\sqrt{n}}.
\] 
By Gronwall inequality for all $t \in [-T_1,T_2]$ we have
\begin{equation}\label{eq:80}
g \leq \frac{C(R)}{\sqrt{n}} \exp(C(R)(T_2+T_1)) \leq \frac{C(R)}{\sqrt{n}}.
\end{equation}
Assume by contradiction that there exist $t$ and $x$ such that
\[
w(t,x)\neq v(t,x).
\]
By continuity of $v$ and $w$, there exists $\eps>0$ such that (for $n>|x|$) we have
\[
 g(t)= \norm{(w - v)\sqrt{\chi_n}}^2_{L^2}>\eps.
\]
Since $\eps>0$ is independant of $n$, we obtain a contradiction with \eqref{eq:80} when $n$ is large enough. Therefore for all $t$ and $x$, we have
\[
v(t,x)=w(t,x),
\]
which concludes the proof.
% We set
% \[
% \Pi = \left\{x \in \mathbb{R} \text{such that} |w-v|>0 \right\}
% \]
% Let $x \in \Pi$ and set $\varepsilon = |w-v|$. Moreover all function in $X^k(\mathbb{R})$ have representation uniformly continous then  there exists $\delta >0$ such that for all $y \in B(x,\delta)$ we have  
% \[
% |w-v|(y) > \frac{\varepsilon}{2}
% \]
% It follows, from \eqref{eq:80} that
% \begin{align*}
% \frac{C(R)}{\sqrt{n}} \geq \int_{\mathbb{R}} |w-v|^2 \chi_n  dx \geq \int_{B(x,\delta)}|w-v|^2 \chi_n  dx \geq \left(\frac{\varepsilon}{2}\right)^2 \mu(B(x,\delta)) =\left(\frac{\varepsilon}{2}\right)^2 \mu(B(0,\delta)) := C,
% \end{align*}
% for all $n>2\delta + |x|$, because $\chi_n(y) = 1$ for $|y| < n$.\\
% Let $n \to +\infty$, we have the contracdition.
% Therefore for all $t \in [T_1,T_2]$ we have
% \[
% w = \partial v + \frac{i}{2} (|v|^2 -1)v + 1.
% \] 
% By the fact that $[T_1,T_2]$ is arbitrait interval on $(-T_{min},T^{max})$ we obtain the result for all $t \in (-T_{min},T^{max})$.
\end{proof}

\subsection{From the system to the equation}
\label{sec:from-system-equation}

With  Proposition \ref{pro:15} in hand, we give the proof of Theorem \ref{local well posed on Zhidkov space}.
\begin{proof}[Proof of Theorem \ref{local well posed on Zhidkov space}.]
  We start by defining $v_0$ by
  \[
v_0 = \partial u_0 + \frac{i}{2} |u_0|^2 u_0  \in X^3(\mathbb{R}).
\]
From Proposition \ref{pro:10} there exists a unique maximal solution $(u,v) \in C((-T_{min},T^{max}), X^3(\mathbb{R})\times X^3(\R)) \cap C^1((-T_{min},T^{max}),X^1(\mathbb{R})\times X^1(\R))$ of the system \eqref{eq:5}   associated with $(u_0,v_0)$. From Proposition \ref{pro:15}, for all $t \in (-T_{min},T^{max})$ we have 
  \begin{equation}
v=\partial u + \frac{i}{2} |u|^2u.
\label{eq:4}
\end{equation}
It follows that 
\[
Lu = -iu^2 \overline{v}+ \frac{1}{2}|u|^4u = -iu^2\partial\overline{u},
\]
and therefore $u$ is a solution of \eqref{eq:1} on $(-T_{min},T^{max})$. Furthermore 
\[
u \in C((-T_{min},T^{max}), X^3(\R)) \cap C^1((-T_{min},T^{max}),X^1(\R)).
\]
To obtain the desired regularity on $u$, we observe that, since $v$ has the same regularity as $u$, and verifies \eqref{eq:4}, we have
% On the other hand, since 
% \[
% v \in C((-T_{min},T^{max}), X^3) \cap C^1((-T_{min},T^{max}),X^1)
% \]
% we have
\[
\partial u = v-\frac{i}{2}|u|^2u \in C((-T_{min},T^{max}), X^3(\R)) \cap C^1((-T_{min},T^{max}),X^1(\R))
\]
This implies that
\[
u \in C((-T_{min},T^{max}), X^4(\R)) \cap C^1((-T_{min},T^{max}),X^2(\R)).
\]
This proves the existence part of the result. Uniqueness is a direct consequence from Proposition \ref{pro:from-eq-to-syst} and Proposition \ref{pro:10}.

% For the uniqueness of solution of equation \eqref{eq:1}, assume $\tilde u \in C((-T_{min},T^{max}), X^4) \cap C^1((-T_{min},T^{max}),X^2)$ is other solution of equation \eqref{eq:1}. Set
% \[
% w = \partial \tilde u + \frac{i}{2}  \tilde u(| \tilde u|^2-1)+1
% \]
% We obtain that $( \tilde u),  w \in C((-T_{min},T^{max}), X^3) \cap C^1((-T_{min},T^{max}),X^1)$ is solution of system \eqref{eq:5}. Since $(u,v) \in C((-T_{min},T^{max}), X^3) \cap C^1((-T_{min},T^{max}),X^1) $ is also solution of system \eqref{eq:5} and the uniqueness solution of system \eqref{eq:5} we obtain that
% \[
% u \equiv \tilde u
% \]  
To prove the blow-up alternative, assume that $T^{max}<\infty$. Then from Proposition \ref{pro:10} we have
\[
\lim_{t \to T^{max}}(\norm{u(t)}_{X^3(\mathbb{R})} + \norm{v(t)}_{X^3(\mathbb{R})}) = \infty
\] 
On the other hand, from the differential identity \eqref{eq:4} we obtain
\[
\lim_{t \to T^{max}}(\norm{u(t)}_{X^3(\mathbb{R})} + \norm{\partial u(t)}_{X^3(\mathbb{R})}) = \infty.
\]
It follows that
\[
\lim_{t \to T^{max}}\norm{u(t)}_{X^4(\mathbb{R})} = \infty.
\]
Finally, we establish the continuity with respect to the initial data.
Take a subinterval $[T_1,T_2] \subset (-T_{min},T^{max})$, and a sequence $(u^n_0) \in X^4(\R)$ such that  $u^n_0 \to u_0$ in $X^4$. Let $u_n$ be the solution of \eqref{eq:1} associated with $u_0^n$ and define $v_n$ by
  \begin{equation}
v_n = \partial u_n + \frac{i}{2}|u_n|^2u_n .\label{eq:7}
\end{equation}
By Proposition \ref{pro:10} the couple $(u_n,v_n)$ is the unique maximal solution of system \eqref{eq:5} in 
\[
 C((-T_{min},T^{max}), X^3(\R)\times X^3(\R)) \cap  C^1((-T_{min},T^{max}),X^1(\R)\times X^1(\R)).
 \]
Moreover, we have 
\begin{equation}\label{eq:100}
\lim_{n \to +\infty} \left(\norm{u_n-u}_{L^{\infty}([T_1,T_2],X^3)} + \norm{v_n-v}_{L^{\infty}([T_1,T_2],X^3)}\right) = 0
\end{equation}
Since $v$ and $v_n$ verify the differential identity \eqref{eq:7}, we have
\begin{equation*}
\partial(u_n-u) = (v_n-v) - \frac{i}{2} \left(|u_n|^2u_n - |u|^2 u\right).
\end{equation*}
Therefore we have
\[
\lim_{n \to +\infty} \norm{u_n-u}_{L^{\infty}([T_1,T_2],X^4)} = 0,
\]
which completes the proof.
\end{proof}

\section{Results on the space $\phi+H^k(\mathbb{R})$ for $\phi \in X^k(\mathbb{R})$}
\label{section 2}

In this section, we give the proof of Theorem \ref{thm local well posedness on phi plus H2 space} and Theorem \ref{local well posedness on phi plus H1 space}. For $k \geq 1$, let $\phi \in X^k(\R)$.

\subsection{The local well posedness on $\phi+H^2(\R)$}

\subsubsection{From the equation to the system}
\label{sec:from-equation-system-1}

\begin{proposition}
  \label{pro:from-eq-to-syst-2}
  If $u$ is a solution of \eqref{eq:1} then for $v$ defined by
  \begin{equation}
    \label{eq:8}
v = \partial u + \frac{i}{2}u(|u|^2-|\phi|^2)+\phi.
\end{equation}
the couple $(u,v)$ verifies the system
\begin{equation}
  \label{eq:m10}
  \left\{
\begin{aligned}
Lu&=Q_1(u,v,\phi) ,\\
Lv&= Q_2(u,v,\phi) ,
\end{aligned}
\right.
\end{equation}
where $Q_1$ and $Q_2$ are given by
\begin{equation}
  \label{eq:9}
  \begin{aligned}
    Q_1(u,v,\phi)&=-iu^2(\overline{v}-\overline{\phi})+\frac{1}{2}u|u|^2(|u|^2-|\phi|^2),\\
    Q_2(u,v,\phi)&=\partial^2\phi + u^2|u|^2(\overline{v}-\overline{\phi}) +i\overline{u}\left((v-\phi)^2-i(v-\phi)u(|u|^2-|\phi|^2)\right) \\
&+\frac{1}{2}|u|^4(v-\phi)-\frac{1}{2}u^2|\phi|^2(\overline{v}-\overline{\phi})-\frac{i}{2}\partial^2(|\phi|^2)u-i\partial(|\phi|^2)(v-\phi)-\frac{i}{2}u(|u|^2-|\phi|^2).
  \end{aligned}
\end{equation}

The functions $Q_1(u,v,\phi)$ and $Q_2(u,v,\phi)$ are polynomials of degree at most $5$ in $u,v,\phi,\partial\phi,\partial^2\phi$. Set $\tilde u= u-\phi$, $\tilde v=v-\phi$. The functions $u,v$ are solutions of the system \eqref{eq:m10} if and only if the functions $\tilde u,\tilde v$ are solutions of the following system
\begin{equation}
  \label{eq:m20}
  \left\{
\begin{aligned}
L\tilde u &= \tilde Q_1(\tilde u,\tilde v,\phi),  \\
L\tilde v &= \tilde Q_2(\tilde u,\tilde v,\phi),  
\end{aligned}
\right.
\end{equation}
where
\begin{align*}
\tilde Q_1(\tilde u,\tilde v,\phi)&:= Lu-L\phi= Q_1(u,v,\phi)-\partial^2\phi = Q_1(\tilde u+\phi,\tilde v+\phi,\phi) -\partial^2\phi,\\
\tilde Q_2(\tilde u,\tilde v,\phi)&:= Lv-L\phi= Q_2(u,v,\phi)-\partial^2\phi.
\end{align*}
\end{proposition}

\begin{proof}

Since   $u$ is a solution of \eqref{eq:1} we have 
\begin{equation*}
Lu = -iu^2\partial\overline{u} = -iu^2\left(\overline{v}+\frac{i}{2}\overline{u}(|u|^2-|\phi|^2)-\overline{\phi}\right) 
= -iu^2(\overline{v}-\overline{\phi})+\frac{1}{2}u|u|^2(|u|^2-|\phi|^2),
\end{equation*}
which gives us the first equation in \eqref{eq:m10}. 
On the other hand, applying $L$ to $v$, we obtain
\begin{align} 
Lv &= \partial(Lu) + \frac{i}{2}L(|u|^2u)-\frac{i}{2}L(|\phi|^2u) + L(\phi) \\
&= \partial(-iu^2\partial\overline{u})+\frac{i}{2}L(u^2\overline{u})-\frac{i}{2}L(|\phi|^2u)+ \partial^2\phi \\
& = -2iu|\partial u|^2 - iu^2\partial^2\overline{u}+\partial^2\phi+\frac{i}{2}L(u^2\overline{u})-\frac{i}{2}L(|\phi|^2u). \label{eq:m1}
\end{align}
As in the proof of Proposition \ref{pro:from-eq-to-syst}, we use \eqref{eq:L-on-product} to get
\begin{align}
L(|\phi|^2u) &= L(u)|\phi|^2+L(|\phi|^2)u+2\partial(|\phi|^2)\partial u \\
&= |\phi|^2(-iu^2\partial\overline{u})+\partial^2(|\phi|^2)u+2\partial u \partial(|\phi|^2).\label{eq:m2}
\end{align}
Recall from \eqref{eq:i-L-u-cube} in the proof of Proposition \eqref{eq:1} that
\begin{equation*}
\frac{i}{2} L(|u|^2u) = u^2 \partial\overline{u} |u|^2 + i \overline{u} (\partial u)^2 + i u^2 \partial^2 \overline{u} + \frac{1}{2} \partial u |u|^4 + 2i u |\partial u|^2.
\end{equation*}
Combining the previous identities, we obtain
\begin{align*}
Lv &= -2iu|\partial u|^2-iu^2\partial^2\overline{u}+\partial^2\phi + u^2\partial\overline{u}|u|^2 + i\overline{u}(\partial u)^2 \\&+iu^2\partial^2\overline{u}+\frac{1}{2}\partial u|u|^4+2iu|\partial u|^2 -\frac{i}{2}(-iu^2\partial\overline{u}|\phi|^2+\partial^2(|\phi|^2)u+2\partial u\partial(|\phi|^2))  \\
   &= \partial^2\phi +u^2\partial\overline{u}|u|^2 + i\overline{u}(\partial u)^2+ \frac{1}{2}|u|^4\partial u-\frac{1}{2}u^2\partial\overline{u}|\phi|^2-\frac{i}{2}\partial^2(|\phi|^2)u-i\partial(|\phi|^2)\partial u.
\end{align*}
Using the differential identity \eqref{eq:8}, we get
     \begin{align*}
Lv &=\partial^2\phi +u^2|u|^2\left(\overline{v}-\overline{\phi}+\frac{i}{2}\overline{u}(|u|^2-|\phi|^2)\right) + i\overline{u}\left(v-\phi-\frac{i}{2}u(|u|^2-|\phi|^2)\right)^2\\
&+\frac{1}{2}|u|^4\left(v-\phi-\frac{i}{2}u(|u|^2-|\phi|^2)\right) - \frac{1}{2}u^2|\phi|^2\left(\overline{v}-\overline{\phi}+\frac{i}{2}\overline{u}(|u|^2-|\phi|^2)\right)-\frac{i}{2}\partial^2(|\phi|^2)u \\
&-i\partial(|\phi|^2)\left(v-\phi-\frac{i}{2}u(|u|^2-|\phi|^2)\right) \\
&= \bigg(\partial^2\phi + u^2|u|^2(\overline{v}-\overline{\phi}) +i\overline{u}\left((v-\phi)^2-i(v-\phi)u(|u|^2-|\phi|^2)\right) \\
&+\frac{1}{2}|u|^4(v-\phi)-\frac{1}{2}u^2|\phi|^2(\overline{v}-\overline{\phi})-\frac{i}{2}\partial^2(|\phi|^2)u-i\partial(|\phi|^2)(v-\phi)-\frac{i}{2}u(|u|^2-|\phi|^2)\bigg) \\
&+ \bigg(\frac{i}{2}u|u|^4(|u|^2-|\phi|^2)-\frac{i}{4}u|u|^2(|u|^2-|\phi|^2)^2-\frac{i}{4}u|u|^4(|u|^2-|\phi|^2) \\
       &-\frac{i}{4}u|u|^2|\phi|^2(|u|^2-|\phi|^2)\bigg).
     \end{align*}
     Observing that
     \begin{equation*}
\frac{i}{2}u|u|^4(|u|^2-|\phi|^2)-\frac{i}{4}u|u|^2(|u|^2-|\phi|^2)^2-\frac{i}{4}u|u|^4(|u|^2-|\phi|^2) 
-\frac{i}{4}u|u|^2|\phi|^2(|u|^2-|\phi|^2)= 0,
\end{equation*}
we obtain the second equation in \eqref{eq:m10}.
% Hence,
% \[
% Lv=J_1:=Q_2(u,v,\phi)
% \]
% We obtain the following system equation of $u,v$
% \begin{equation}
% \label{eq:m10}
% \begin{cases}
% Lu= -iu^2(\overline{v}-\overline{\phi})+\frac{1}{2}u|u|^2(|u|^2-|\phi|^2)=Q_1(u,v,\phi) ,\\
% Lv= Q_2(u,v,\phi) ,\\
% u(0)=u_0,v(0)=v_0.
% \end{cases}
% \end{equation}
% where $Q_1(u,v,\phi)$ is a polynomial of degree small or equal $5$ of $u,v,\phi$ and $Q_2(u,v,\phi)$ is a polynomial of $u,v,\phi,\partial\phi,\partial^2\phi$.
% Set $\tilde u= u-\phi$, $\tilde v=v-\phi$. The functions $u,v$ are the solutions of the system \eqref{eq:m10} if only if the functions $\tilde u,\tilde v$ are solutions of the following system
% \begin{equation}
% \label{eq:m20}
% \begin{cases}
% L\tilde u = Lu-\partial^2\phi= Q_1(u,v,\phi)-\partial^2\phi = Q_1(\tilde u+\phi,\tilde v+\phi,\phi) -\partial^2\phi := \tilde Q_1(\tilde u,\tilde v,\phi),  \\
% L\tilde v = Lv-\partial^2\phi = Q_2(u,v,\phi)-\partial^2\phi := \tilde Q_2(\tilde u,\tilde v,\phi),  \\
% \tilde u(0) = \tilde u_0:=u_0-\phi,\tilde v(0) = \tilde v_0:=v_0-\phi.
% \end{cases}
% \end{equation}
\end{proof}

\subsubsection{Resolution of the system}
\label{sec:resolution-system-1}

From similar arguments to the one used for the proof of Proposition \ref{pro:10}, we may obtain the following local well-posedness result.
\begin{proposition}\label{pro:m207}
  Let $k \geq 1$, $\phi \in X^{k+2}$, $\tilde u_0,\tilde v_0 \in H^k(\mathbb{R})$.
There exist $T_{min},T^{max} > 0$ and a unique maximal solution $(\tilde u,\tilde v)$ of the system \eqref{eq:m20} such that
$\tilde u,\tilde v \in C((-T_{min},T^{max}),H^k(\mathbb{R})) \cap C^1((-T_{min},T^{max}),H^{k-2}(\mathbb{R}))$.  Furthermore the following properties are satisfied.
\begin{itemize}
\item \emph{Blow-up alternative.} If $T^{max}$ (resp. $T_{min}$)$< +\infty$ then
\[
\lim_{t \to T^{max}\text{(resp. } T_{min}\text{)}}(\norm{\tilde u}_{H^k}+\norm{\tilde v}_{H^k}) = \infty.
\]
\item \emph{Continuity with respect to the initial data.} If $\tilde u^n_0,\tilde v^n_0 \in H^k(\mathbb{R})$  are such that
  \[
    \norm{\tilde u^n_0-\tilde u_0}_{H^k}+\norm{\tilde v^n_0-\tilde v_0}_{H^k} \to 0
  \]
  then  for any subinterval $[T_1,T_2] \subset (-T_{min},T^{max})$ the associated solution $(\tilde u^n,\tilde v^n)$ of \eqref{eq:m20} satisfies
\[
\lim_{n \to +\infty}\left(\norm{\tilde u^n-\tilde u}_{L^{\infty}([T_1,T_2],H^k)}+\norm{\tilde v^n-\tilde v}_{L^{\infty}([T_1,T_2],H^k)}\right) = 0.
\]
\end{itemize}
\end{proposition}

% \begin{proof}
% We have $\tilde u,\tilde v \in C((-T_{min},T^{max}),H^k(\mathbb{R})) \cap C^1((-T_{min},T^{max}),H^{k-2}(\mathbb{R}))$ are the solutions of the system \eqref{eq:m20} if only if $\tilde u,\tilde v \in C((-T_{min},T^{max}),H^k(\mathbb{R}))$ are the solutions of the following equations
% \begin{equation}
% \label{eq:m100}
% (\tilde u,\tilde v) = S(t)(\tilde u_0,\tilde v_0) -i\int_0^t S(t-s)\tilde Q(\tilde u,\tilde v,\phi)(s) \, ds.
% \end{equation}
% where $\tilde Q(\tilde u,\tilde v,\phi) = \left(\tilde Q_1(\tilde u,\tilde v,\phi),\tilde Q_2(\tilde u,\tilde v,\phi)\right)$. From $\phi \in X^{k+2}$ and the fact that $\tilde Q_1,\tilde Q_2$ are the polynomial of $u,v,\phi,\partial \phi,\partial^2\phi$ we have $\tilde Q(.,.,\phi)$ is Lipchitz continuous on bounded set of $H^k(\mathbb{R})\times H^k(\mathbb{R})$ in $H^k(\mathbb{R})\times H^k(\mathbb{R})$. By using similarly the arguments in the proposition \ref{pro:10} we obtain the desired results.
% \end{proof}

\subsubsection{Preservation of a differential identity}
\label{sec:pres-diff-ident-1}

Given well-posedness of the system \eqref{eq:m10}, we need to show preservation of the differential identity to go back to \eqref{eq:1}. This is the object of the following proposition.

\begin{proposition}
\label{cor:pvt}
Let $\phi \in X^4(\R)$ and $\tilde u_0, \tilde v_0 \in H^2(\mathbb{R})$ such that the condition
\begin{equation}
\label{condition new}
\tilde v_0 = \partial\tilde u_0+\frac{i}{2}(\tilde u_0+\phi)(|\tilde u_0+\phi|^2-|\phi|^2)+\partial\phi
\end{equation}
is verified. Then the associated solutions $\tilde u,\tilde v$ obtained in Proposition \ref{pro:m207} also satisfy \eqref{condition new} for all $t \in (-T_{min},T^{max})$.   
\end{proposition}
\begin{proof}
We define 
%Proposition \ref{cor:pvt} is a direct consequence of the following proposition. 

%\begin{proposition}\label{pro:new 001}
 % Let $k \geq 3$, $\phi \in X^{k+2}(\mathbb{R})$, $u_0,v_0 \in X^k(\mathbb{R})$ such that the condition
%  \[
 %   v_0=\partial u_0 + \frac{i}{2}u_0(|u_0|^2-|\phi|^2)+\phi
  %\]
  %holds. Then the associated solution $u,v \in C((-T_{min},T^{max}),X^k(\mathbb{R})) \cap C^1((-T_{min},T^{max}),X^{k-2}(\mathbb{R}))$ of the system  \eqref{eq:m10} verifies also for all $t \in (-T_{min},T^{max})$ the condition
%\[
%v=\partial u+\frac{i}{2}u(|u|^2-|\phi|^2)+\phi.
%\]
%\end{proposition}

%\begin{proof}

\begin{equation}\label{eq:diff-id-w-2}
\tilde w = \partial \tilde u+\frac{i}{2}(\tilde u+\phi)(|\tilde u+\phi|^2-|\phi|^2)+\partial \phi.
\end{equation}
%We want to prove that $v=w$. 
Set $u=\tilde u+\phi$, $v=\tilde v+\phi$, $w = \tilde w+\phi$. We have
\[
w=\partial u+\frac{i}{2}u(|u|^2-|\phi|^2)+\phi.
\]
Since $\tilde u,\tilde v$ is a solution of \eqref{eq:m20}, we have$u,v$ is a solution of \eqref{eq:m10}. 
\begin{equation*}
Lu=-iu^2(\overline{v}-\overline{w}) - i u^2(\overline{w} - \overline{\phi}) +\frac{1}{2}u|u|^2(|u|^2-|\phi|^2) = -iu^2(\overline{v}-\overline{w}) + H,
\end{equation*}
where we have defined
\[
H=- i u^2(\overline{w} - \overline{\phi}) +\frac{1}{2}u|u|^2(|u|^2-|\phi|^2) .
\]
Applying $L$ to $w$ and using \eqref{eq:L-u-cube} and the previously expression obtained for $Lu$, we get

\begin{align*}
Lw &= \partial(Lu)+\frac{i}{2}L(|u|^2u)-\frac{i}{2}L(|\phi|^2u)+L(\phi) \\
&= \partial(Lu) + \frac{i}{2}\left(2L(u)|u|^2+2\overline{u}(\partial u)^2 + 2u^2\partial^2\overline{u}- u^2\overline{L(u)}+4u|\partial u|^2 \right)\\
& - \frac{i}{2}\left(|\phi|^2L(u)+u\partial^2(|\phi|^2) + 2\partial u \partial(|\phi|^2)\right) \\
&=  \partial\left(-iu^2(\overline{v}-\overline{w})\right)+\partial H \\ 
&+\frac{i}{2} \left(2H|u|^2-2iu^2|u|^2(\overline{v}-\overline{w})+2\overline{u}(\partial u)^2+2u^2\partial^2\overline{u}-u^2\left(i\overline{u}^2(v-w)+\overline{H}\right)+4u|\partial u|^2\right) \\
&-\frac{i}{2}\left(-iu^2(\overline{v}-\overline{w})|\phi|^2+|\phi|^2H+u\partial^2(|\phi|^2)+2 \partial u\partial(|\phi|^2)\right) \\
&= -i \partial \left(u^2(\overline{v}-\overline{w})\right) +u^2|u|^2(\overline{v}-\overline{w})+\frac{1}{2}|u|^4(v-w)-\frac{1}{2}u^2(\overline{v}-\overline{w}) + K,
\end{align*}
where $K$ depends on $u$, $w$ and $\phi$ but not on $v$ and is given by
\[
K = \partial^2\phi+u^2\partial\overline{u}|u|^2+i\overline{u}(\partial u)^2+\frac{1}{2}|u|^4\partial u-\frac{1}{2}u^2\partial\overline{u}|\phi|^2-i\partial(|\phi|^2)\partial u.
\]
Using the differential identity \eqref{eq:diff-id-w-2} to replace $\partial u$, we obtain for $K$ the following
%
 %By the fact that $w$ can performance in term of $u$, so do $K$ and by the same previous calculations we have
\begin{align*}
K&=\partial^2\phi +u^2|u|^2\left(\overline{w}-\overline{\phi}+\frac{i}{2}\overline{u}(|u|^2-|\phi|^2)\right) + i\overline{u}\left(w-\phi-\frac{i}{2}u(|u|^2-|\phi|^2)\right)^2\\
&+\frac{1}{2}|u|^4\left(w-\phi-\frac{i}{2}u(|u|^2-|\phi|^2)\right) - \frac{1}{2}u^2|\phi|^2\left(\overline{w}-\overline{\phi}+\frac{i}{2}u(|u|^2-|\phi|^2)\right)-\frac{i}{2}\partial^2(|\phi|^2)u \\
&-i\partial(|\phi|^2)\left(w-\phi-\frac{i}{2}u(|u|^2-|\phi|^2)\right) \\
&= \partial^2\phi + u^2|u|^2(\overline{w}-\overline{\phi})+i\overline{u}\left((w-\phi)^2-i(w-\phi)u(|u|^2-|\phi|^2)\right) \\
&+\frac{1}{2}|u|^4(w-\phi)-\frac{1}{2}|\phi|^2|u|^2(\overline{w}-\overline{\phi})-\frac{i}{2}\partial^2(|\phi|^2)u-i\partial(|\phi|^2)(w-\phi)-\frac{i}{2}u(|u|^2-|\phi|^2).
\end{align*}  
As a consequence, we arrive  for $L(w)-L(v)$ at the following expression:
\begin{align*}
Lw - Lv &= -i \partial \left(u^2(\overline{v}-\overline{w})\right) +u^2|u|^2(\overline{v}-\overline{w})+\frac{1}{2}|u|^4(v-w)-\frac{1}{2}u^2(\overline{v}-\overline{w}) + (K-L(v))\\
&= -iu^2\partial(\overline{v}-\overline{w})+A(v-w)+B(\overline{v}-\overline{w}),
\end{align*}
where $A,B$ are polynomials in $u,v,\phi,\partial\phi,\partial^2\phi$. It implies that
\begin{equation}\label{eq new new}
L (\tilde w-\tilde v) = -i(\tilde u+\phi)^2\partial(\overline{\tilde v} -\overline{\tilde w})+A(\tilde v-\tilde w)+B(\overline{\tilde v} -\overline{\tilde w}).
\end{equation}
Multiplying two sides of \eqref{eq new new} by $\overline{\tilde w}-\overline{\tilde v}$, taking the imaginary part, and integrating over space with integration by part for the first term of right hand side of \eqref{eq new new}, we obtain
\[
\frac{d}{dt}\norm{\tilde w-\tilde v}^2_{L^2} \lesssim (\norm{\tilde u+\phi}_{L^{\infty}}+\norm{\partial\tilde u+\partial \phi}_{L^{\infty}}+\norm{A}_{L^{\infty}}+\norm{B}_{L^{\infty}})\norm{\tilde w-\tilde v}^2_{L^2}.
\]
By Gr\"onwall's inequality we obtain 
\[
\norm{\tilde w-\tilde v}^2_{L^2} \leq \norm{\tilde{w}(0)-\tilde{v}(0)}^2_{L^2} \times exp(C\int_0^t(\norm{\tilde u+\phi}_{L^{\infty}}+\norm{\partial\tilde u+\partial \phi}_{L^{\infty}}+\norm{A}_{L^{\infty}}+\norm{B}_{L^{\infty}}) \,ds).
\] 
Using the fact that $\tilde{w}(0)=\tilde{v}(0)$, we obtain $\tilde w=\tilde v$, for all $t$. It implies that
\[
 \tilde v=\partial \tilde u+\frac{i}{2}(\tilde u+\phi)(|\tilde u+\phi|^2-|\phi|^2)+\partial\phi.
 \]
This complete the proof of Proposition \ref{cor:pvt}. 
\end{proof}

\subsubsection{From the system to the equation}
\label{sec:from-system-equation-1}

With local well-posedness of the system and preservation of the differential identity in hand, we may now go back to the original equation and finish the proof of Theorem \ref{thm local well posedness on phi plus H2 space}.

\begin{proof}[Proof of Theorem \ref{thm local well posedness on phi plus H2 space}.]
Let $\phi \in X^4(\R)$. We define $v_0\in X^1(\R)$, $\tilde u_0\in H^2(\R)$ and $\tilde v_0\in H^1(\R)$ in the following way:
\[
  v_0=\partial u_0+\frac{i}{2}u_0(|u_0|^2-|\phi|^2)+\phi,\quad \tilde u_0=u_0-\phi,\quad \text{ and } \tilde v_0=v_0-\phi .
\]
We have
\[
\tilde v_0 =\partial \tilde u_0+\frac{i}{2}(\tilde u_0+\phi)(|\tilde u_0+\phi|^2-|\phi|^2)+\partial\phi.
\]
From Proposition \ref{pro:m207} there exists a unique maximal solution $\tilde u,\tilde v \in C((-T_{min},T^{max}),H^1(\mathbb{R})) \cap C^1((-T_{min},T^{max}),H^{-1}(\mathbb{R}))$ of \eqref{eq:m20}. 
%By Proposition \ref{proconserve}, we have 
Let $\tilde u^n_0 \in H^3(\mathbb{R})$ be such that
\[
  \norm{\tilde u^n_0-\tilde u_0}_{H^2(\mathbb{R})} \to 0
\]
as $n \to \infty$. Define $\tilde v^n_0 \in H^2(\mathbb{R})$ by 
\[
\tilde v^n_0 = \partial \tilde u^n_0 +\frac{i}{2}(\tilde u^n_0+\phi)(|\tilde u^n_0+\phi|^2-|\phi|^2)+\partial\phi.
\]
From Proposition \ref{pro:m207}, there exists a unique solution maximal solution.
\[
 \tilde u^n,\tilde v^n \in C((-T^n_{min},T_{max}^n),H^2(\mathbb{R})) \cap C^1((-T^n_{min},T^n_{max}),L^2(\mathbb{R}))
\]
of the system \eqref{eq:m20}. Let $[-T_1,T_2] \subset (-T_{min},T^{max})$ be any closed interval. From \cite[proposition 4.3.7]{CaHa98}, for $n \geq N_0$ large enough, we have $[-T_1,T_2] \subset (-T^n_{min},T^n_{max})$. By Proposition \ref{cor:pvt}, for $n \geq N_0$, $t \in [-T_1,T_2]$, we have
\[
\tilde v^n = \partial \tilde u^n+\frac{i}{2}(\tilde u^n+\phi)(|\tilde u^n+\phi|^2-|\phi|^2)+\partial\phi.
\]  
By Pproposition \ref{pro:m207}, we have
\[
\lim_{n \to +\infty}\sup_{t\in[T_1,T_2]}(\norm{\tilde u^n(t)-\tilde u(t)}_{H^1(\mathbb{R})}+\norm{\tilde v^n-\tilde v(t)}_{H^1(\mathbb{R})}) \to 0.
\] 
Letting $n \to +\infty$, we obtain that for all $t \in [-T_1,T_2]$, and then for all $t \in (-T_{min},T^{max})$:

\[
\tilde v = \partial\tilde u+\frac{i}{2}(\tilde u+\phi)(|\tilde u+\phi|^2-|\phi|^2)+\partial\phi.
\]
It follows that
\[
\partial\tilde u \in C((-T_{min},T^{max}),H^1(\mathbb{R})) \cap C^1((-T_{min},T^{max}),H^{-1}(\mathbb{R})).
\]
Hence we have
\[
\tilde u \in C((-T_{min},T^{max}),H^2(\mathbb{R})) \cap C^1((-T_{min},T^{max}),L^2(\mathbb{R}))
\]
Define $u\in \phi+ C((-T_{min},T^{max}),\phi+H^2(\mathbb{R})) \cap C^1((-T_{min},T^{max}),L^2(\mathbb{R}))$ by
\[
u=\phi+u.
\]
and define $v\in \phi+C((-T_{min},T^{max}),\phi+H^1(\mathbb{R})) \cap C^1((-T_{min},T^{max}),H^{-1}(\mathbb{R}))$ by
\[
v= \tilde v+\phi= \partial u+\frac{i}{2}u(|u|^2-|\phi|^2)+\phi.
\]
Since $\tilde u,\tilde v$ are solution of system \eqref{eq:m20},  $u,v$ are solutions of the system \eqref{eq:m10}. Therefore, 
\[
Lu=Q_1(u,v)=Q_1\left(u,\partial u+\frac{i}{2}u(|u|^2-|\phi|^2)+\phi\right)=-iu^2\partial\overline{u}.
\]
This establishes the existence of a solution to \eqref{eq:1}. To prove uniqueness, assume that $U \in \phi+C((-T_{min},T^{max}),H^2(\mathbb{R})) \cap C^1((-T_{min},T^{max}),L^2(\mathbb{R}))$ is another solution of \eqref{eq:1}. Set $V=\partial U+\frac{i}{2}U(|U|^2-|\phi|^2)+\phi$, and $\tilde U = U-\phi$, $\tilde V=V-\phi$. We see that $\tilde U,\tilde V \in C((-T_{min},T^{max}),H^1(\mathbb{R})) \cap C^1((-T_{min},T^{max}),H^{-1}(\mathbb{R}))$ are solutions of the system \eqref{eq:m20}. From the uniqueness statement in Proposition \ref{pro:m207} we obtain $\tilde U=\tilde u$. Hence, $u=U$, which proves uniqueness.
The blow-up alternative and continuity with respect to the initial data are proved using similar arguments as in the proof of Theorem \ref{local well posed on Zhidkov space}. This completes the proof of Theorem \ref{thm local well posedness on phi plus H2 space}.
\end{proof}

\subsection{The local well posedness on $\phi+H^1(\R)$}
In this section, we give the proof of Theorem \ref{local well posedness on phi plus H1 space}, using the method of Hayashi and Ozawa \cite{HaOz94}. As in Section \ref{sec:from-equation-system-1}, we work with the system \eqref{eq:m20}.

\subsubsection{Resolution of the system}
Since we are working in the less regular space $\phi+H^1(\R)$, we cannot use Proposition \ref{pro:m207}. Instead, we establish the following result using Strichartz estimate. 
\begin{proposition}
\label{pro:n1}
Consider the system \eqref{eq:m20}. Let $\phi \in X^2(\mathbb{R})$, $\tilde u_0,\tilde v_0 \in L^2(\mathbb{R})$. There exists $R>0$ such that if $\norm{\tilde u_0}_{L^2}+\norm{\tilde v_0}_{L^2}<R$ then there exists $T>0$ and a unique solution $\tilde u,\tilde v$ of the system \eqref{eq:m20} verifiying
\[
\tilde u,\tilde v \in C([-T,T],L^2) \cap L^4([-T,T],L^{\infty}).
\]
Moreover, we have the following continuous dependence on initial data property: If $(\widetilde{u}^n_0,\widetilde{v}^n_0) \in L^2(\R) \times L^2(\R)$  is a sequence such that $\norm{\widetilde{u}^n_0}_2+\norm{\widetilde{v}^n_0}_2<R$ and $\norm{\widetilde{u}^n_0-u_0}_2+\norm{\widetilde{v}^n_0-v_0}_2 \rightarrow 0$ then the associated solutions $(\tilde{u}^n,\tilde{v}^n)$ such that   
\[
\norm{\tilde{u}^n-\tilde u}_{L^{\infty}L^2 \cap L^4L^{\infty}}+\norm{\tilde{v}^n-\tilde v}_{L^{\infty}L^2 \cap L^4L^{\infty}} \rightarrow 0,
\]
where $T$ is the time of existence of $\tilde u$, $\tilde v$ and we have used the following notation:
\[
  L^{\infty}L^2=L^{\infty}([-T,T],L^2(\mathbb{R})),\quad L^4L^{\infty} = L^4([-T,T],L^{\infty}(\mathbb{R}))
\]
and the norm on  $L^{\infty}L^2 \cap L^4L^{\infty}$ is defined, as usual for  the intersection of two Banach spaces, as the sum of the norms on each space.
\end{proposition}

\begin{proof}
  
Let $\widetilde{Q}_1, \widetilde{Q}_2$ as in system \eqref{eq:m20}. By direct calculations, we have
\begin{align}
\tilde Q_1(\tilde u,\tilde v,\phi) &=  -i(\tilde u+\phi)^2\overline{\tilde v} +\frac{1}{2}(\tilde u+\phi)|(\tilde u+\phi)|^2(|(\tilde u+\phi)|^2-|\phi|^2)-\partial^2\phi \label{eq:haha}
,\\ 
\tilde Q_2(\tilde u,\tilde v,\phi) &= (\tilde u+\phi)^2|\tilde u+\phi|^2\overline{\tilde v}+i(\overline{\tilde u}+\overline{\phi})\left((\tilde v)^2-i\tilde v(\tilde u+\phi)(|\tilde u+\phi|^2-|\phi|^2)\right) \nonumber \\
&+\frac{1}{2}|\tilde u+\phi|^4\tilde v-\frac{1}{2}(\tilde u+\phi)|\phi|^2\overline{\tilde v}-\frac{i}{2}\partial^2(|\phi|^2)(\tilde u+\phi)-i\partial(|\phi|^2)\tilde v \nonumber \\
&-\frac{i}{2}(\tilde u+\phi)(|\tilde u+\phi|^2-|\phi|^2).
\label{eq:hehe}
\end{align}
Consider the following problem 
\begin{equation}
\label{eq:n1}
(\tilde u,\tilde v) = S(t)(\tilde u_0,\tilde v_0) - i \int_0^t S(t-s)\tilde Q(\tilde u,\tilde v,\phi) \, ds
\end{equation}
where $\tilde Q = (\tilde Q_1,\tilde Q_2)$. Let 
\[
\Phi(\tilde u,\tilde v) = S(t)(\tilde u_0,\tilde v_0) - i \int_0^t S(t-s)\tilde Q \, ds. 
\]
Assume that $\norm{\tilde u_0}_{L^2(\mathbb{R})}+\norm{\tilde v_0}_{L^2(\mathbb{R})} \leq \frac{R}{4}$ for $R>0$ small enough.
For $T>0$ we define the space $X_{T,R}$ by
\[
X_{T,R} = \left\{(\tilde u,\tilde v) \in C([-T,T],L^2(\mathbb{R})) \cap L^4([-T,T],L^{\infty}(\mathbb{R})) : \norm{\tilde u}_{L^{\infty}L^2 \cap L^4L^{\infty}} + \norm{\tilde v}_{L^{\infty}L^2 \cap L^4L^{\infty}}  \leq R \right\}.
\] 
We are going to prove that for $R,T$ small enough the map $\Phi$ is a contraction from $X_{T,R}$ to $X_{T,R}$.

We first prove that for $R,T$ small enough, $\Phi$ maps $X_{T,R}$ into $X_{T,R}$. Let $(\tilde u,\tilde v) \in X_{T,R}$. By Strichartz estimates we have
\begin{align*}
\norm{\Phi(\tilde u,\tilde v)}_{L^{\infty}L^2 \cap L^4L^{\infty}} & \lesssim \norm{(\tilde u_0,\tilde v_0)}_{L^2 \times L^2} + \norm{\tilde Q}_{L^1L^2 \times L^1L^2} ,\\ 
& \lesssim \frac{R}{4} + (\norm{\tilde Q_1}_{L^1L^2}+\norm{\tilde Q_2}_{L^1L^2}). 
\end{align*}
We have 
\begin{align*}
\norm{\tilde Q_1}_{L^1L^2} & \lesssim \norm{|\tilde u+\phi|^2\tilde v}_{L^1L^2}+\norm{|\tilde u+\phi|^3(|\tilde u|^2+|\tilde u||\phi|)}_{L^1L^2} + \norm{\partial^2\phi}_{L^1L^2} \\
& \lesssim \norm{\tilde v}_{L^2L^2}\norm{\tilde u+\phi}_{L^4L^{\infty}}^2 + \norm{\tilde u+\phi}_{L^4L^{\infty}}^3(\norm{\tilde u}_{L^4L^{\infty}}\norm{\tilde u}_{L^{\infty}L^2}+\norm{\tilde u}_{L^{\infty}L^2}\norm{\phi}_{L^4L^{\infty}}) + \norm{\partial^2\phi}_{L^1L^2} \\
& \lesssim \norm{\tilde v}_{L^{\infty}L^2}(2T)^{\frac{1}{2}}(\norm{\tilde u}_{L^4L^{\infty}}+\norm{\phi}_{L^4L^{\infty}})^2  \\ 
& + (\norm{\tilde u}_{L^4L^{\infty}}+\norm{\phi}_{L^4L^{\infty}})^3\left(\norm{\tilde u}_{L^4L^{\infty}}\norm{\tilde u}_{L^{\infty}L^2}+\norm{\tilde u}_{L^{\infty}L^2}\norm{\phi}_{L^{\infty}}(2T)^{\frac{1}{4}}\right) + \norm{\partial^2\phi}_{L^1L^2} \\
& \lesssim (2T)^{\frac{1}{2}}R(R+\norm{\phi}_{L^{\infty}}(2T)^{\frac{1}{4}})^2+(R+\norm{\phi}_{L^{\infty}}(2T)^{\frac{1}{4}})^3\left(R^2+R\norm{\phi}_{L^{\infty}}(2T^{\frac{1}{4}})\right)+ (2T)\norm{\partial^2\phi}_{L^2(\mathbb{R})}  \\
& \lesssim \frac{R}{4},
\end{align*} 
for $T,R$ small enough. Since $\tilde Q_2$ contains polynomial of order at most 5, we also have
\[
\norm{\tilde Q_2}_{L^1L^2} \lesssim \frac{R}{4}
\]
for $T,R$ small enough. Therefore, for $T,R$ small enough, we have
\[
\norm{\Phi(\tilde u,\tilde v)}_{(L^{\infty}L^2\cap L^4L^{\infty})^2} \leq \frac{3R}{4} < R.
\]
Hence, $\Phi$ maps from $X_{T,R}$ into itself.

We now show that for $T,R$ small enough, the map $\Phi$ is a contraction from $X_{T,R}$ to itself.\\
Indeed, let $(u_1,v_1),(u_2,v_2) \in X_{T,R}$. By Strichartz estimates we have
\begin{align*}
\norm{\Phi(u_1,v_1)-\Phi(u_2,v_2)}_{L^{\infty}L^2 \cap L^4L^{\infty}} & = \lVert{\int_0^t S(t-s)\left(\tilde Q(u_1,v_1)-\tilde Q(u_2,v_2)\right)\,ds\rVert}_{(L^{\infty}L^2 \cap L^4L^{\infty})^2} ,\\
& \lesssim \norm{\tilde Q_1(u_1,v_1)-\tilde Q_1(u_2,v_2)}_{L^1L^2} + \norm{\tilde Q_2(u_1,v_1)-\tilde Q_2(u_2,v_2)}_{L^1L^2}. 
\end{align*}
Using the same kind of arguments as before we obtain that $\Phi$ is a contraction on $X_{T,R}$. Therefore, using the Banach fixed-point theorem, there exist $T>0$ and a unique solution $\tilde u,\tilde v \in C([-T,T],L^2{\mathbb{R}}) \cap L^4([-T,T],L^{\infty}(\mathbb{R}))$ of the problem \eqref{eq:n1}. As above, we see that if $h,k \in C([-T,T],L^2{\mathbb{R}}) \cap L^4([-T,T],L^{\infty}(\mathbb{R}))$ then $\tilde Q_1(h,k),\tilde Q_2(h,k) \in L^1([-T,T], L^2(\R))$. By \cite[Proposition 4.1.9]{CaHa98}, $\tilde u,\tilde v \in C([-T,T],L^2{\mathbb{R}}) \cap L^4([-T,T],L^{\infty}(\mathbb{R}))$ solves \eqref{eq:n1} if only if $\tilde u,\tilde v$ solves \eqref{eq:m20}. Thus, we prove the existence of solution of \eqref{eq:m20}. The uniqueness of solution of \eqref{eq:m20} is obtained by the uniqueness of solution of \eqref{eq:n1}.      

It is remains to prove the continuous dependence on initial data. Assume that $(u^n_0,v^n_0) \in L^2(\R) \times L^2(\R)$ is such that 
\[
\norm{u^n_0-\tilde{u}_0}_{L^2(\R)}+\norm{v^n_0-\tilde{v}_0}_{L^2(\R)} \rightarrow 0,
\] 
as $n \rightarrow \infty$. In particular, for $n$ large enough, we have
\[
\norm{u^n_0}_{L^2(\R)} +\norm{v^n_0}_{L^2(\R)}<R.
\] 
There exists a unique maximal solution $(u^n,v^n)$ of system \eqref{eq:m20}, and we may assume that for $n$ large enough, $(u^n,v^n)$ is defined on $[-T,T]$. Assume that $T$ small enough such that
\begin{equation}\label{eq:smaller}
\norm{\tilde u}_{L^{\infty}L^2\cap L^4L^{\infty}}+\norm{\tilde v}_{L^{\infty}L^2\cap L^4L^{\infty}}+\mathop{\sup}\limits_{n}(\norm{u^n}_{L^{\infty}L^2\cap L^4L^{\infty}}+\norm{v_n}_{L^{\infty}L^2\cap L^4L^{\infty}}) \leq 2R.
\end{equation}
The functions $(\tilde u,\tilde v)$ are solutions of the following system
\[
(\tilde u,\tilde v) = S(t)(\tilde u_0,\tilde v_0) -i\int_0^t S(t-s)(\tilde Q_1(\tilde u,\tilde v,\phi) ,\tilde Q_2(\tilde u,\tilde v,\phi) ).
\]
Similarly, $(u^n,v^n)$ are solutions of the following system 
\[
(u^n,v^n) = S(t)(u^n_0,v^n_0) - i\int_0^t S(t-s)(\tilde Q_1(u^n,v^n,\phi),\tilde Q_2(u^n,v^n,\phi)).
\]
Hence,
\begin{align*}
&(u^n-u,v^n-v) \\& = S(t)(u^n_0-\tilde u_0,v^n_0-\tilde v_0) - i\int_0^t S(t-s)(\tilde Q_1(\tilde u,\tilde v,\phi)-\tilde Q_1(u^n,v^n,\phi),\tilde Q_2(\tilde u,\tilde v,\phi)-\tilde Q_2(u^n,v^n,\phi)).
\end{align*}
Using Strichartz estimates and \eqref{eq:smaller}, for all $t \in [-T,T]$ and $R$, $T$ small enough, we have
\begin{align*}
&\norm{u^n-\tilde u}_{L^{\infty}L^2 \cap L^4L^{\infty}} +\norm{v^n-\tilde v}_{L^{\infty}L^2 \cap L^4L^{\infty}}\\
& \lesssim \norm{u^n_0-\tilde u_0}_{L^2} + \norm{v^n_0-\tilde v_0}_{L^2}\\
&+ \norm{\tilde Q_1(\tilde u,\tilde v,\phi)-\tilde Q_1(u^n,v^n,\phi)}_{L^1L^2}+\norm{\tilde Q_2(\tilde u,\tilde v,\phi)-\tilde Q_2(u^n,v^n,\phi))}_{L^1L^2} \\
& \lesssim \norm{u^n_0-\tilde u_0}_{L^2} + \norm{v^n_0-\tilde v_0}_{L^2} \\
&+ R(\norm{u^n-\tilde u}_{L^{\infty}L^2 \cap L^4L^{\infty}} +\norm{v^n-\tilde v}_{L^{\infty}L^2 \cap L^4L^{\infty}}).
\end{align*}
For $R<\frac{1}{2}$ small enough, we have
\[
\frac{1}{2}(\norm{u^n-\tilde u}_{L^{\infty}L^2 \cap L^4L^{\infty}} +\norm{v^n-\tilde v}_{L^{\infty}L^2 \cap L^4L^{\infty}}) \leq \norm{\tilde u_0-u^n_0}_{L^2(\mathbb{R})}+\norm{\tilde v_0-v^n_0}_{L^2(\mathbb{R})} .
\]
Letting $n \to +\infty$ we obtain the desired result.
\end{proof}

By similar arguments we obtain the following result in higher regularity.
\begin{proposition}\label{pro:abc}
Consider the system \eqref{eq:m20}. Let $\phi \in X^4(\mathbb{R})$ and $\tilde u_0, \tilde v_0 \in H^2(\mathbb{R})$ such that $\norm{\tilde u_0}_{H^2(\mathbb{R})}+\norm{\tilde v_0}_{H^2(\mathbb{R})} < R$ small enough. Then, there exist $T = T(R)$ and a unique solution $\tilde u, \tilde v \in C([-T,T],H^2(\mathbb{R})) \cap L^4([-T,T],W^{2,\infty}(\mathbb{R}))$. 
\end{proposition}

\subsubsection{Preservation of a differential identity}

By Proposition \ref{cor:pvt} the solutions obtained by Proposition \ref{pro:abc} satisfy the following  property.
\begin{proposition}\label{pro:fize}
Let $\phi \in X^4(\mathbb{R})$, $\tilde u_0, \tilde v_0 \in H^2(\mathbb{R})$ such that 
\[
\tilde v_0=\partial \tilde u_0 + \frac{i}{2}(\tilde u_0+\phi)(|\tilde u_0+\phi|^2-|\phi|^2)+\partial \phi.
\]
Then the associated solutions $\tilde u, \tilde v$ of \eqref{eq:m20} satisfy the following condition for all $t \in [-T,T]$
\[
\tilde v=\partial \tilde u + \frac{i}{2}(\tilde u+\phi)(|\tilde u+\phi|^2-|\phi|^2)+\partial\phi.
\]
\end{proposition}

\subsubsection{From the system to the equation}
In this section, we prove Theorem \ref{local well posedness on phi plus H1 space} by getting back to the equation.
\begin{proof}[Proof of Theorem \ref{local well posedness on phi plus H1 space}.] Let $\phi \in X^4(\R)$ such that $\norm{\partial\phi}_{L^2}$ is small enough. Let $u_0 \in \phi+H^1(\R)$ be such that $\norm{u_0-\phi}_{H^1}$ small enough. Set $v_0=\partial u_0+\frac{i}{2}u_0(|u_0|^2-|\phi|^2)+\phi$, $\tilde u_0=u_0-\phi$ and $\tilde v_0=v_0-\phi$. We have
\[
\tilde v_0=\partial\tilde u_0+\frac{i}{2}(\tilde u_0+\phi)(|\tilde u_0+\phi|^2-|\phi|^2)+\partial \phi.
\]
Furthermore, $\tilde u_0 \in H^1(\mathbb{R})$, $\tilde v_0 \in L^2(\mathbb{R})$. We have
\[
\norm{\tilde u_0}_{L^2(\mathbb{R})}+\norm{\tilde v_0}_{L^2(\mathbb{R})} \lesssim \norm{\tilde u_0}_{H^1(\mathbb{R})} +\norm{\partial\phi}_{L^2},
\]
which is small enough by the assumption. By Proposition \ref{pro:n1}, there exists $T>0$ and a unique solution $\tilde u, \tilde v \in C([-T,T],L^2(\mathbb{R})) \cap L^4([-T,T],L^{\infty})$ of the system \eqref{eq:m20}. %By Proposition \ref{proconservesss}, we obtain
%\begin{equation}\label{eq:convergence}
%\tilde v=\partial \tilde u+\frac{i}{2}(\tilde u+\phi)(|\tilde u+\phi|^2-|\phi|^2)+\partial\phi.
%\end{equation}
Let $u^n_0 \in H^3(\mathbb{R})$ be such that  $\norm{u^n_0}_{H^3(\mathbb{R})}$ small enough and $\norm{u^n_0-\tilde u_0}_{H^1(\mathbb{R})} \to 0$ as $n \to +\infty$. Set
\[
v^n_0 = \partial u^n_0 + \frac{i}{2}(u^n_0+\phi)(|u^n_0+\phi|^2-|\phi|^2)+\partial\phi.
\]
We have
\[
\norm{\tilde u^n_0}_{H^2(\mathbb{R})}+\norm{\tilde v_0}_{H^2(\mathbb{R})} \lesssim \norm{\tilde u_0}_{H^3(\mathbb{R})} +\norm{\partial\phi}_{H^2},
\] 
which is small enough by the assumption. Let $(u^n,v^n)$ be the $H^2(\R)$ solution of the system \eqref{eq:m20} obtained by Proposition \ref{pro:abc} with data $(u^n_0,v^n_0)$. By Proposition \ref{pro:fize} we have
\begin{equation}\label{eq:relative}
v^n=\partial u^n+\frac{i}{2}(u^n+\phi)(|u^n+\phi|^2-|\phi|^2)+\partial\phi.
\end{equation}
Furthermore, we have
\[
\norm{u^n_0-\tilde u_0}_{L^2(\mathbb{R})} + \norm{v^n_0-\tilde v_0}_{L^2(\mathbb{R})} \to 0.
\]
From the continuous dependence on the initial data obtained in Proposition \ref{pro:n1}, $(u^n,v^n),(\tilde u,\tilde v)$ are solutions of the system \eqref{eq:m20} on $[-T,T]$ for $n$ large enough, and
\[
\norm{u^n-\tilde u}_{L^{\infty}L^2 \cap L^4L^{\infty}} + \norm{v^n-\tilde v}_{L^{\infty}L^2 \cap L^4L^{\infty}} \to 0
\]
as $n \to \infty$. Let $n \to \infty$ on the two sides of \eqref{eq:relative}, we obtain for all $t \in [-T,T]$ 
\begin{equation}\label{eq:convergence}
\tilde v=\partial \tilde u+\frac{i}{2}(\tilde u+\phi)(|\tilde u+\phi|^2-|\phi|^2)+\partial\phi,
\end{equation}
which make sense in $H^{-1}(\mathbb{R})$. From \eqref{eq:convergence} we see that $\partial \tilde u \in C([-T,T],L^2(\mathbb{R}))$ and \eqref{eq:convergence} makes sense in $L^2(\mathbb{R})$. Then $\tilde u \in C([-T,T],H^1(\mathbb{R})) \cap L^4([-T,T],L^{\infty})$. By the Sobolev embedding of $H^1(\mathbb{R})$ in $L^{\infty}(\mathbb{R})$ we obtain that 
\begin{align*}
\norm{(\tilde u+\phi)(|\tilde u+\phi|^2-|\phi|^2)}_{L^4L^{\infty}} & \lesssim \norm{|\tilde u+\phi||\tilde u||\tilde u|+|\phi|} \lesssim \norm{\tilde u}_{L^4L^{\infty}}(\norm{\tilde u}_{L^{\infty}L^{\infty}}+\norm{\phi}_{L^{\infty}L^{\infty}})^2 \\
& \lesssim \norm{\tilde u}_{L^4L^{\infty}}(\norm{\tilde u}_{L^{\infty}H^1}+\norm{\phi}_{L^{\infty}L^{\infty}})^2 <\infty.
\end{align*}
Hence, $(\tilde u+\phi)(|\tilde u+\phi|^2-|\phi|^2) \in L^4L^{\infty}$. From \eqref{eq:convergence} we obtain that $\partial\tilde u \in L^4L^{\infty}$ which implies $\tilde u \in L^4([-T,T],W^{1,\infty}(\mathbb{R}))$. Set $u = \tilde u+\phi$, $v= \tilde v + \phi$, then $u-\phi \in C([-T,T],H^1(\mathbb{R})) \cap L^4([-T,T],W^{1,\infty}(\mathbb{R}))$ and $v-\phi \in C([-T,T],L^2(\mathbb{R})) \cap L^4([-T,T],L^{\infty}(\mathbb{R}))$. Moreover,
\[
v=\partial u+\frac{i}{2}u(|u|^2-|\phi|^2)+\phi.
\] 
Since $u,v$ solve \eqref{eq:m10}, we have
\[
Lu= Q_1(u,v,\phi) = Q_1\left(u,\partial u+\frac{i}{2}u(|u|^2-|\phi|^2)+\phi,\phi\right)=-iu^2\partial\overline{u}.
\]
The existence of a solution of the equation \eqref{eq:1} follows. To prove the uniqueness property, assume that $U \in C([-T,T],\phi+H^1(\mathbb{R})) \cap L^4([-T,T],\phi+W^{1,\infty}(\mathbb{R}))$ is an other solution of the equation \eqref{eq:1}. Set $V=\partial U+\frac{i}{2}U(|U|^2-|\phi|^2)+\phi$ and $\tilde U=U-\phi$, $\tilde V=V-\phi$. Hence $\tilde U \in C([-T,T],H^1(\mathbb{R})) \cap L^4([-T,T],W^{1,\infty}(\mathbb{R}))$ and $\tilde V \in C([-T,T],L^2(\mathbb{R})) \cap L^4([-T,T],L^{\infty}(\mathbb{R}))$. Moreover, $\tilde U, \tilde V$ is solution of the system \eqref{eq:m20}. By the uniqueness of solutions of \eqref{eq:m20}, we obtain that $\tilde U=\tilde u$. Hence, $u=U$, which complete the proof.
\end{proof}

\begin{remark}\label{remark Cauchy problem}
In \cite{Ge06}, G\'erard gives the proof of local well posedness of solutions of the Gross-Pitaevskii in energy space, using some properties of the energy space. More precisely, he proved that there exists a unique maximal solution $u \in C((-T_{min},T^{max}),\mathcal{E})$ of the problem
\begin{equation}\label{Gross Pitaske equation}
u(t)=S(t)u_0-i\int_0^t S(t-s)(u(s)(|u(s)|^2-1))\,ds,
\end{equation}
where $S(t)$ is the Schr\"odinger group, $u_0\in \mathcal{E}$ is given and $\mathcal{E}$ is the energy space which is defined by
\[
\mathcal{E}:=\left\{u\in H^1_{loc}(\R): \partial u \in L^2(\R), |u|^2-1 \in L^2(\R)\right\}.
\]
The proof of Gérard is in dimension $2$ and $3$. We can give a simple proof of this result in one dimension. Indeed, we see that $u_0\in\mathcal{E}\subset X^1(\R)$, then, it is easy to prove that there exists a unique maximal solution $u \in C((-T_{min},T^{max}),X^1(\R))$. Set 
\[
w(t,x)=u(t,x)-u_0(x).
\]
Consider the following problem
\begin{equation}\label{modfined of Gross Pitaske equation}
w(t)=-i\int_0^t S(t-s)((u_0+w(s))(|u_0+w(s)|^2-1))\, ds.
\end{equation}
We can check that the function $P : H^1(\R) \rightarrow H^1(\R)$ defined by
\[
P(w)=(u_0+w)(|u_0+w|^2-1),
\]
is Lipschitz continuous on bounded set of $H^1(\R)$. Thus, by elementary arguments, there exists unique maximal solution $w \in C((-T_{min},T^{max}),H^1(\R))$ of the Cauchy problem \eqref{modfined of Gross Pitaske equation}. It implies that there exists a unique maximal solution $u \in C((-T_{min},T^{max}),u_0+H^1(\R))$ of the Cauchy problem \eqref{Gross Pitaske equation}. Using the fact that
\[
u_0+H^1(\R) \subset \mathcal{E} \subset X^1(\R),
\] 
we obtain that there exists a unique maximal solution $u \in C((-T_{min},T^{max}),\mathcal{E})$ of the Cauchy problem \eqref{Gross Pitaske equation}.
\end{remark}

\section{Conservation of the mass, the energy and the momentum}
In this section, we prove Theorem \ref{thm conservation law special case}. Let $q_0 \in \R$ and $u \in q_0+H^2(\R)$ be a solution of \eqref{eq:1}. Let $\chi$ and $\chi_R$ be the functions defined as in \eqref{eq of chi} and \eqref{eq of chi R}. We have
\begin{align}\label{converge to 0 as R converge infity}
\norm{\partial\chi_R}_{L^2(\R)} &= \left(\int_{\R}\frac{1}{R}\chi'\left(\frac{x}{R}\right)\right)^{\frac{1}{2}} = \frac{1}{R^{\frac{1}{2}}} \norm{\chi}_{L^2(\R)} \to 0 \, \text{as} \, R \to +\infty.
\end{align}
By the continuous depend on initial data property of solution, we can assume that 
\[
u \in C((-T_{min},T_{max}),q_0+H^3(\R)).
\]
It is enough to prove \eqref{eq:conservation mass special}, \eqref{eq:conservation energy special} and \eqref{eq:conservation momentum special} for any closed interval $[-T_0,T_1] \in (-T_{min},T_{max})$. Let $T_0>0$, $T_1>0$ be such that $[-T_0,T_1] \subset (-T_{min},T_{max})$. Let $M>0$ be defined by
\[
M =\mathop{\sup}\limits_{t\in [-T_0,T_1]}\norm{u-q_0}_{H^2(\R)}.
\]

\subsection{Conservation of mass}

Multiply the two sides of \eqref{eq:1} with $\overline{u}$ and take imaginary part to obtain 
\begin{align*}
\Re(u_t\overline{u}) + \Im(\partial^2u\overline{u}) + \Re(|u|^2u\partial\overline{u}) &=0.
\end{align*} 
This implies that

\begin{align*}
0 &= \frac{1}{2}\partial_t (|u|^2) + \partial(\Im(\partial u\overline{u})) + \frac{1}{4} \partial(|u|^4) \\
&= \frac{1}{2}\partial_t (|u|^2-q_0^2) + \partial(\Im(\partial u\overline{u})) + \frac{1}{4} \partial(|u|^4-q_0^4).
\end{align*}

By multiplying the two sides with $\chi_R$, integrating on space, and integrating by part we have
\begin{align}
0 &= \partial_t \int_{\R} \frac{1}{2} (|u|^2-q_0^2)\chi_R  dx - \int_{\R} \Im(\partial u\overline{u}) \partial \chi_R - \int_{\R} \frac{(|u|^4-q_0^4)}{4} \partial \chi_R  dx\\
&=  \partial_t \int_{\R} \frac{1}{2} (|u|^2-q_0^2)\chi_R  dx - \int_{\R}\left(\Im(\partial u\overline{u})+\frac{1}{4}(|u|^4-q_0^4)\right)\partial\chi_R  dx . \label{eq r}
\end{align}

Denote the second term of \eqref{eq r} by $K$, using \eqref{converge to 0 as R converge infity}, we have
\[
|K| \leq \norm{\Im(\partial u\overline{u})+\frac{1}{4}(|u|^4-q_0^4)}_{L^2}\norm{\partial \chi_R}_{L^2} \lesssim C(M)\frac{1}{R^{\frac{1}{2}}} \rightarrow 0 \text{ as } R \rightarrow \infty.
\]
Thus, by integrating from $0$ to $t$ and taking $R$ to infinity, using the assumption $|u_0|^2-q_0^2 \in L^1(\R)$ we obtain

\begin{align}
\mathop{\lim}\limits_{R \rightarrow \infty} \int_{\R} \frac{1}{2} (|u|^2-q_0^2)\chi_R  dx &= \mathop{\lim}\limits_{R \rightarrow \infty} \int_{\R} \frac{1}{2} (|u_0|^2-q_0^2)\chi_R  dx =\frac{1}{2}\int_{\R}(|u_0|^2-q_0^2)\, dx. \label{eq z}
\end{align}
Thus, we obtain the conservation of the mass \eqref{eq:conservation mass special}. 
\subsection{Conservation of energy}
Now, we prove the conservation of the energy. Since $u$ solves \eqref{eq:1}, after elementary calculations, we have 

\begin{align}
\label{eq:e^3q^3}
\partial_t(|\partial u|^2) &= \partial\left(2\Re(\partial u\partial_t  \overline{u})+\Re(u^2(\partial\overline{u})^2)-|\partial u|^2|u|^2-|u|^4\Im(\overline{u}\partial u)\right)+|u|^4\partial\Im(\overline{u} \partial u) + 2 \Im(|u|^2\partial\overline{u}u_t).
\end{align}

Recall that we have

\begin{align}
\label{eq:eeqq}
\partial\Im(\partial u\overline{u}) &= -\frac{1}{2}\partial_t(|u|^2) -\frac{1}{4}\partial(|u|^4).
\end{align}

Moreover, we have

\begin{align*}
\partial_t \Im(|u|^2u\partial\overline{u}) &= 4\Im(u_t|u|^2\partial\overline{u}) + \partial\Im(|u|^2u\partial_t\overline{u}).
\end{align*}
%&=\Im(\partial_t(|u|^2)u\partial\overline{u})+\Im(|u|^2\partial\overline{u}u_t)+\Im(|u|^2u\partial_t\partial\overline{u}) \\
%&= \partial_t(|u|^2)\Im(u\partial\overline{u})+\Im(|u|^2\partial\overline{u}u_t)+\partial\Im(|u|^2u\partial_t\overline{u}) -\Im(\partial_t\overline{u}\partial(|u|^2 u)) \\
%&= 2\Re(u_t\overline{u})\Im(u\partial\overline{u})+ \Im(|u|^2\partial\overline{u}u_t) + \partial\Im(|u|^2 u\partial_t\overline{u})- \Im(\partial_t\overline{u}(2|u|^2\partial u+u^2\partial\overline{u})) \\
%&= 2\Re(u_t\overline{u})\Im(u\partial\overline{u})+ \Im(|u|^2\partial\overline{u}u_t) + \partial\Im(|u|^2 u\partial_t\overline{u})-2|u|^2\Im(\partial_t\overline{u}\partial u) + \Im(u_t\overline{u}^2\partial u) \\
%&= 2\Re(u_t\overline{u})\Im(u\partial\overline{u})+ \Im(|u|^2\partial\overline{u}u_t) + \partial\Im(|u|^2 u\partial_t\overline{u}) \\ 
%&+ 2\Im(|u|^2 u_t\partial\overline{u})+\Re(u_t\overline{u})\Im(\overline{u}\partial u)+\Im(u_t\overline{u})\Re(\overline{u}\partial u) \\
%&= 3\Im(u_t|u|^2\partial\overline{u})+\partial\Im(|u|^2u\partial_t\overline{u})+ \left(2\Re(u_t\overline{u})\Im(u\partial\overline{u})+\Re(u_t\overline{u})\Im(\overline{u}\partial u)+\Re(\overline{u}\partial u)\Im(u_t\overline{u})\right) \\
%&= 4\Im(u_t|u|^2\partial\overline{u}) + \partial\Im(|u|^2u\partial_t\overline{u}).
%\end{align*}

It follows that

\begin{align}
\label{eq:eq^4}
2\Im(|u|^2u_t\partial\overline{u}) &= \frac{1}{2}\left(\partial_t\Im(|u|^2u\partial\overline{u}) - \partial\Im(|u|^2u\partial_t\overline{u})\right).
\end{align}

From \eqref{eq:e^3q^3}, \eqref{eq:eeqq} and \eqref{eq:eq^4} we have

\begin{align*}
\partial_t(|\partial u|^2)&= \partial\left(2\Re(\partial u\partial_t\overline{u}) +\Re(u^2(\partial\overline{u})^2)-|u|^2|\partial u|^2-|u|^4\Im(\partial u\overline{u})-\frac{1}{2}\Im(|u|^2u\partial_t\overline{u})\right) \\
&+\frac{1}{2}\partial_t\Im(|u|^2u\partial\overline{u}) -\frac{1}{8}\partial(|u|^8)-\frac{1}{6}\partial_t(|u|^6).
\end{align*}

Hence,

\begin{align*}
&\partial_t\left(|\partial u|^2-\frac{1}{2}\Im((|u|^2u-q_0^3)\partial\overline{u})+\frac{1}{6}(|u|^6-q_0^6)\right) \\
&=\partial\left(2\Re(\partial u\partial_t\overline{u}) +\Re(u^2(\partial\overline{u})^2)-|u|^2|\partial u|^2-|u|^4\Im(\partial u\overline{u})-\frac{1}{2}\Im(|u|^2u\partial_t\overline{u}) -\frac{1}{8}(|u|^8-q_0^8) \right)+\frac{1}{2}q_0^3\Im\partial_t\partial\overline{(u-q_0)}.
\end{align*}

By multiplying the two sides with $\chi_R$ then integrating in space and integrating by part we obtain 

\begin{align*}
&\partial_t\int_{\R}\left(|\partial u|^2-\frac{1}{2}\Im((|u|^2u-q_0^3)\partial\overline{u})+\frac{1}{6}(|u|^6-q_0^6)\right)\chi_R \, dx \\
&=-\int_{\R}\partial\chi_R \left(2\Re(\partial u\partial_t\overline{u}) +\Re(u^2(\partial\overline{u})^2)-|u|^2|\partial u|^2-|u|^4\Im(\partial u\overline{u})-\frac{1}{2}\Im(|u|^2u\partial_t\overline{u}) -\frac{1}{8}(|u|^8-q_0^8)\right)\, dx \\
&-\frac{q_0^3}{2}\Im\partial_t\int_{\R}\overline{(u-q_0)}\partial\chi_R \, dx.
\end{align*}

Integrating from $0$ to $t$ we obtain

\begin{align}
&\int_{\R}\left(|\partial u|^2-\frac{1}{2}\Im((|u|^2u-q_0^3)\partial\overline{u})+\frac{1}{6}(|u|^6-q_0^6)\right)\chi_R \, dx \quad \label{eq c}\\
&-\int_{\R}\left(|\partial u_0|^2-\frac{1}{2}\Im((|u_0|^2u_0-q_0^3)\partial\overline{u_0})+\frac{1}{6}(|u_0|^6-q_0^6)\right)\chi_R \, dx \quad \label{eq d}\\
&= \int_0^t\int_{\R}\partial\chi_R \left(2\Re(\partial u\partial_t\overline{u}) +\Re(u^2(\partial\overline{u})^2)-|u|^2|\partial u|^2-|u|^4\Im(\partial u\overline{u}) \right. \nonumber \\
&\quad \left.-\frac{1}{2}\Im(|u|^2u\partial_t\overline{u}) -\frac{1}{8}(|u|^8-q_0^8)\right)\, dx \, ds \label{eq a1}\\
&-\frac{q_0^3}{2}\left(\Im\int_{\R}\overline{(u-q_0)}\partial\chi_R \, dx-\Im\int_{\R}\overline{(u_0-q_0)}\partial\chi_R \, dx\right)\quad \label{eq b}.
\end{align}

Denoting the term \eqref{eq a1} by $A$, using \eqref{converge to 0 as R converge infity}, we have

\begin{align}
|A| &\leq \norm{\partial\chi_R}_{L^2}\norm{2\Re(\partial u\partial_t\overline{u}) +\Re(u^2(\partial\overline{u})^2)-|u|^2|\partial u|^2-|u|^4\Im(\partial u\overline{u})\\
&-\frac{1}{2}\Im(|u|^2u\partial_t\overline{u}) -\frac{1}{8}(|u|^8-q_0^8)}_{L^2}\\
&\lesssim C(M)\norm{\partial\chi_R}_{L^2} \rightarrow 0 \text{ as } R \rightarrow \infty. \label{eq converge n}
\end{align} 

Moreover, using \eqref{converge to 0 as R converge infity} again, we have 
\begin{align}
&\left|\Im\int_{\R}\overline{(u-q_0)}\partial\chi_R \, dx\right| \leq \norm{u-q_0}_{L^2}\norm{\partial\chi_R}_{L^2} \lesssim C(M)\norm{\partial\chi_R}_{L^2} \rightarrow 0 \text{ as } R \rightarrow \infty. \label{eq converge l}\\
&\left|\Im\int_{\R}\overline{(u_0-q_0)}\partial\chi_R \, dx\right| \leq \norm{u_0-q_0}_{L^2}\norm{\partial\chi_R}_{L^2} \lesssim C(M)\norm{\partial\chi_R}_{L^2} \rightarrow 0 \text{ as } R \rightarrow \infty. \label{eq converge m}
\end{align}

To deal with the term \eqref{eq c}, we need to divide into two terms. First, using $u\in q_0+H^3(\R)$, as $R \rightarrow \infty$, we have

\begin{align}\label{eq xz}
\int_{\R} \left(|\partial u|^2-\frac{1}{2}\Im((|u|^2u-q_0^3)\partial\overline{u})\right)\chi_R \, dx \rightarrow \int_{\R} \left(|\partial u|^2-\frac{1}{2}\Im((|u|^2u-q_0^3)\partial\overline{u})\right) \, dx.
\end{align} 

Second, by easy calculations, we have

\begin{align}
&\frac{1}{6}\int_{\R}(|u|^6-q_0^6)\chi_{R} \, dx  \label{eq g1}\\
&= \frac{1}{6}\int_{\R}\left[(|u|^2-q_0^2)(|u|^4+q_0^2|u_0|^2-2q_0^4)+3q_0^4(|u|^2-q_0^2)\right]\chi_{R}\, dx \nonumber\\
&= \frac{1}{6}\int_{\R}(|u|^2-q_0^2)^2(|u|^2+2q_0^2) \chi_R\,dx \label{eq f}\\
&+\frac{1}{2}\int_{\R}(|u|^2-q_0^2)\chi_{R} \, dx \label{eq g}. 
\end{align}

Denote the term \eqref{eq f} by $B$, we have

\begin{align}\label{eq converge k}
B \rightarrow  \frac{1}{6}\int_{\R}(|u|^2-q_0^2)(|u|^4+q_0^2|u_0|^2-2q_0^4) \, dx \text{ as } R \rightarrow +\infty.
\end{align}

The term \eqref{eq g} is treated using conservation of mass \eqref{eq z}. Finally, we have

\begin{align}
\label{eq xy}
\mathop{\lim}\limits_{R \rightarrow \infty}\frac{1}{6}\int_{\R}(|u|^6-q_0^6)\chi_R\, dx &= \frac{1}{6}\int_{\R}(|u|^2-q_0^2)^2(|u|^2+2q_0^2)\, dx +\frac{q_0^4}{2 }\mathop{\lim}\limits_{R \rightarrow \infty}\int_{\R}(|u_0|^2-q_0^2)\chi_R \, dx.
\end{align}

Combining \eqref{eq xy}, \eqref{eq xz} we have 

\begin{align}
\mathop{\lim}\limits_{R \rightarrow \infty}(\text{ the term } \eqref{eq c}) &=  \int_{\R}|\partial u|^2-\frac{1}{2} \Im(|u|^2u-q_0^3)\partial\overline{u}) \, dx+\frac{1}{6}\int_{\R}(|u|^2-q_0^2)^2(|u|^2+2q_0^2)\, dx \nonumber \\
& +\frac{1}{2 }\mathop{\lim}\limits_{R \rightarrow \infty}\int_{\R}(|u_0|^2-q_0^2)\chi_R \, dx.  \label{eq xyz}
\end{align}

Similarly,

\begin{align}
\mathop{\lim}\limits_{R \rightarrow \infty}(\text{ the term } \eqref{eq d}) &=  \int_{\R}|\partial u_0|^2-\frac{1}{2} \Im(|u_0|^2u_0-q_0^3)\partial\overline{u_0}) \, dx+\frac{1}{6}\int_{\R}(|u_0|^2-q_0^2)^2(|u_0|^2+2q_0^2)\, dx \nonumber \\
& +\frac{1}{2 }\mathop{\lim}\limits_{R \rightarrow \infty}\int_{\R}(|u_0|^2-q_0^2)\chi_R \, dx.  \label{eq xyz1}
\end{align}

Combined \eqref{eq c}-\eqref{eq converge m}, \eqref{eq xyz} and \eqref{eq xyz1}, we have 

\begin{align*}
&\int_{\R}|\partial u|^2-\frac{1}{2} \Im(|u|^2u-q_0^3)\partial\overline{u}) \, dx+\frac{1}{6}\int_{\R}(|u|^2-q_0^2)^2(|u|^2+2q_0^2)\, dx\\
&-\int_{\R}|\partial u_0|^2-\frac{1}{2} \Im(|u_0|^2u_0-q_0^3)\partial\overline{u_0}) \, dx+\frac{1}{6}\int_{\R}(|u_0|^2-q_0^2)^2(|u_0|^2+2q_0^2)\, dx \\
&= 0
\end{align*}

This implies \eqref{eq:conservation energy special}.

\subsection{Conservation of momentum}
Now, we prove \eqref{eq:conservation momentum special}. Multiplying the two sides of \eqref{eq:1} with $-\partial\overline{u}$ and taking real part we obtain 

\begin{align}
0 &= -\Re(iu_t\partial\overline{u}+\partial^2u\partial\overline{u}+iu^2(\partial\overline{u})^2)\nonumber\\
&=\Im(u_t\partial\overline{u})  +\Im(u^2(\partial\overline{u})^2) - \frac{1}{2} \partial(|\partial u|^2). \label{eq:eq}
\end{align}    
Moreover, by elementary calculation, we have

\begin{align*}
\partial_t \Im(u\partial\overline{u}) &= 2 \Im(u_t\partial \overline{u})+ \partial\Im(u\partial_t\overline{u}). 
\end{align*}

Replacing $\Im(u_t\partial \overline{u}) = \frac{1}{2}\left(\partial_t\Im(u\partial\overline{u}) -\partial\Im(u\partial_t\overline{u})\right)$ in \eqref{eq:eq}, we obtain that

\begin{align*}
0 &= \left(\frac{1}{2}\partial_t\Im(u\partial\overline{u}) -\frac{1}{2}\partial\Im(u\partial_t\overline{u})\right) + 2 \Re(u\partial\overline{u})\Im(u\partial\overline{u}) - \frac{1}{2} \partial(|\partial u|^2) \\
&= \partial_t\left[\frac{1}{2}\Im(u\partial\overline{u})-\frac{1}{4}(|u|^4-q_0^4)\right] + \partial\left[\Im(|u|^2u\partial\overline{u})-\frac{1}{2}|\partial u|^2-\frac{1}{6} (|u|^6-q_0^6)\right].
\end{align*}

Multiply two sides by $\chi_R$ and integral on space, using integral by part, we have 

\begin{align}
0 &= \partial_t\int_{\R}\left[\frac{1}{2}\Im(u\partial\overline{u})-\frac{1}{4}(|u|^4-q_0^4)\right] \chi_R  dx - \int_{\R} \left[\Im(|u|^2u\partial\overline{u})-\frac{1}{2}|\partial u|^2-\frac{1}{6}(|u|^6-q_0^6)\right] \partial \chi_R  dx \nonumber\\ 
&= \partial_t\int_{\R}\left[\frac{1}{2}\Im(u\partial\overline{u})-\frac{1}{4}(|u|^2-q_0^2)^2-\frac{1}{2}q_0^2(|u|^2-q_0^2)\right] \chi_R  dx - \int_{\R} \left[\Im(|u|^2u\partial\overline{u})-\frac{1}{2}|\partial u|^2-\frac{1}{6}(|u|^6-q_0^6)\right] \partial \chi_R  dx. 
\label{eq 1234}
\end{align}

Denoting the second term of \eqref{eq 1234} by $D$, we have

\begin{align}\label{estimate 1000}
|D| &\leq \norm{\Im(|u|^2u\partial\overline{u})-\frac{1}{2}|\partial u|^2-\frac{1}{6}(|u|^6-q_0^6)}_{L^2}\norm{\partial\chi_R}_{L^2} \lesssim C(M)\norm{\partial\chi_R}_{L^2} \rightarrow 0 \text{ as } R \rightarrow \infty.
\end{align}

Integrating from $0$ to $t$ the two sides of \eqref{eq 1234} and taking $R$ to infinity, using \eqref{estimate 1000}, we have

\begin{align}
& \int_{\R}\left[\frac{1}{2}\Im(u\partial\overline{u})-\frac{1}{4}(|u|^2-q_0^2)^2\right] \, dx -\frac{q_0^2}{2}\mathop{\lim}\limits_{R \rightarrow \infty}\int_{\R}(|u|^2-q_0^2)\chi_R \, dx \\
&=  \int_{\R}\left[\frac{1}{2}\Im(u_0\partial\overline{u_0})-\frac{1}{4}(|u_0|^2-q_0^2)^2\right] \, dx -\frac{q_0^2}{2}\mathop{\lim}\limits_{R \rightarrow \infty}\int_{\R}(|u_0|^2-q_0^2)\chi_R \, dx. \label{eq zz}
\end{align}

Combined \eqref{eq zz} and \eqref{eq z}, we obtain the conservation of momentum \eqref{eq:conservation momentum special}. which completes the proof of Theorem \ref{thm conservation law special case}.

\section{Stationary solutions}

In this section, we give the proof of Theorem \ref{thm kink soliton}. We start by the following definition of stationary solutions of \eqref{eq:1}.
\begin{definition}\label{definitionofstationarysolution}
The stationary solutions of \eqref{eq:1} are functions $\phi \in X^2(\R)$ satisfying
\begin{equation}\label{eqstationary solution}
\phi_{xx}+i\phi^2\overline{\phi}_x =0.
\end{equation}
\end{definition}

\begin{proof}[Proof of Theorem \ref{thm kink soliton}.] Let $\phi$ be a non constant solution of \eqref{eqstationary solution} such that $m=\mathop{\inf}\limits_{x\in \R}|\phi(x)|>0$. From \eqref{eqstationary solution}, we have $\phi \in X^3(\R)$. Using the assumptions on $\phi$ we can write $\phi$ as
\[
\phi(x)=R(x)e^{i\theta(x)}
\]
where $R>0$ and $R,\theta \in \mathcal{C}^2(\R)$ are real-valued functions. We have 
\begin{align*}
\phi_x &= e^{i\theta}(R_x+i\theta_xR), \\
\phi_{xx}&= e^{i\theta}(R_{xx}+2iR_x\theta_x+iR\theta_{xx}-R\theta_x^2).
\end{align*}
Hence, since $\phi$ satisfies \eqref{eqstationary solution} we obtain
\begin{align*}
0&=(R_{xx}-R\theta_x^2+R^3\theta_x)+i(2R_x\theta_x+R\theta_{xx}+R^2R_x) .
\end{align*}
It is equivalent to
\begin{align}
0&=R_{xx}-R\theta_x^2+R^3\theta_x, \label{eq 1}\\
0&=2R_x\theta_x+R\theta_{xx}+R^2R_x.\label{eq 2}
\end{align}
The equation \eqref{eq 2} is equivalent to
\begin{align*}
0&=\partial_x\left(R^2\theta_x+\frac{1}{4}R^4\right).
\end{align*}
Hence there exists $B \in \R$ such that 
\begin{equation}
B=R^2\theta_x+\frac{1}{4}R^4.
\end{equation}
This implies
\begin{equation}\label{eqofthetax}
 \theta_x=\frac{B}{R^2}-\frac{R^2}{4}. 
\end{equation} 
Substituting the above equality in \eqref{eq 1} we obtain
\begin{align}
0&=R_{xx}-R\left(\frac{B}{R^2}-\frac{R^2}{4}\right)^2+R^3\left(\frac{B}{R^2}-\frac{R^2}{4}\right) \nonumber\\
&= R_{xx}-\frac{B^2}{R^3}-\frac{5R^5}{16}+\frac{3BR}{2}. \label{eqofR}
\end{align}
We prove that the set $V=\left\{x \in \R: \, R_x(x) \neq 0\right\}$ is dense in $\R$. Indeed, assume there exists $x \in \R \setminus \overline{V}$. Thus, there exists $\varepsilon$ such that $B(x,\varepsilon) \in \R \setminus \overline{V}$. It implies that for all $y \in B(x,\varepsilon)$, we have $R_x(y)=0$. Hence, $R$ is a constant function on $B(x,\varepsilon)$. By uniqueness of $C^2$ solution of \eqref{eqofR}, we have $R$ is constant function on $\R$. By \eqref{eqofthetax}, $\theta_x$ is constant. Thus, $\phi(x)$ is of form $Ce^{ikx}$, for some constants $C,K \in \R$. If $k=0$ it is a constant and if $k \neq 0$ it is not in $X^1(\R)$, which contradicts the assumption of $\phi$. From \eqref{eqofR}, we have
\begin{align*}
0&=R_x\left(R_{xx}-\frac{B^2}{R^3}-\frac{5R^5}{16}+\frac{3BR}{2}\right) \\
&=\frac{d}{dx}\left[\frac{1}{2}R_x^2+\frac{B^2}{2R^2}-\frac{5}{96}R^6+\frac{3B}{4}R^2\right]. 
\end{align*}
Hence there exists $a \in \R$ such that
\begin{align*}
a &= \frac{1}{2}R_x^2+\frac{B^2}{2R^2}-\frac{5}{96}R^6+\frac{3B}{4}R^2.
\end{align*}
It is equivalent to
\begin{align*}
0&=R_x^2R^2+B^2-\frac{5}{48}R^8+\frac{3B}{2}R^4-2aR^2,\\
&=\frac{1}{4}[(R^2)_x]^2+B^2-\frac{5}{48}R^8+\frac{3B}{2}R^4-2aR^2.
\end{align*}
Set $k=R^2$. We have
\begin{align}\label{eq 3}
0 &= \frac{1}{4}k_x^2+B^2-\frac{5}{48}k^4+\frac{3B}{2}k^2-2ak.
\end{align}
Differentiating the two sides of \eqref{eq 3} we have
\begin{align*}
0&=k_x\left(\frac{k_{xx}}{2}-\frac{5}{12}k^3+3Bk-2a\right)
\end{align*} 
On the other hand, since $k_x=2R_xR \neq 0$ for a.e $x$ in $\R$, we obtain the following equation for a.e $x$ in $\R$, hence, by continuity of $k$, it is true for all $x$ in $\R$: 
\begin{align}\label{eq 5}
0&=\frac{k_{xx}}{2}-\frac{5}{12}k^3+3Bk-2a.
\end{align}
Now, using Lemma \ref{lm 1} we have $k-2\sqrt{B} \in H^3(\R)$. Combining with \eqref{eq 5} we obtain $a=\frac{4B\sqrt{B}}{3}$. Set $h=k-2\sqrt{B}$. Then from \eqref{eq 5} $h\in H^3(\R)$ solves 
\begin{equation}\label{eq h}
\begin{cases}
0=h_{xx}-\frac{5}{6}h^3-5\sqrt{B}h^2-4Bh , \\
h>m^2-2\sqrt{B}.
\end{cases}
\end{equation} 
The equation \eqref{eq h} can be explicitly solved, and we find
\[
h=\frac{-1}{\sqrt{\frac{5}{72B}}\cosh(2\sqrt{B}x)+\frac{5}{12\sqrt{B}}}.
\]
This implies
\[
k=2\sqrt{B}+h = 2\sqrt{B}+\frac{-1}{\sqrt{\frac{5}{72B}}\cosh(2\sqrt{B}x)+\frac{5}{12\sqrt{B}}}.
\]
Furthermore, using $\theta_x=\frac{B}{k}-\frac{k}{4}$, there exists $\theta_0 \in \R$ such that 
\[
\theta(x)=\theta_0-\int_x^{\infty} \left(\frac{B}{k}-\frac{k}{4}\right)\, dy.
\]
Now, assume that $\phi$ is a solution of \eqref{eqstationary solution} such that $\phi(\infty)=0$. We prove $\phi \equiv 0$ on $\R$. Multiplying the two side of \eqref{eqstationary solution} with $\overline{\phi}$ then taking the imaginary part we obtain 
\[
\partial_x\Im(\phi_x\overline{\phi})+\frac{1}{4}
\partial_x(|\phi|^4) =0
\]
On the other hand, $\phi(\infty)=\phi_x(\infty)=0$ then on $\R$ we have
\begin{equation}\label{eq of phi}
\Im(\phi_x\overline{\phi})+\frac{1}{4}|\phi|^4=0.
\end{equation}
If there exists $y_0$ such that $\phi_x(y_0)=0$ then from \eqref{eq of phi} we have $\phi(y_0)=0$. By the uniqueness of Cauchy problem we obtain $\phi \equiv 0$ on $\R$. Otherwise, $\phi_x$ is non vanishing on $\R$. From now on, we will consider this case. Multiplying the two sides of \eqref{eqstationary solution} with $\overline{\phi_x}$ then taking the real part, we have
\begin{align*}
0&=\Re(\phi_{xx}\overline{\phi_x})-\Im(\phi^2\overline{\phi_x}^2) \\
&= \frac{1}{2}\frac{d}{dx}|\phi_x|^2-2\Re(\phi\overline{\phi_x})\Im(\phi\overline{\phi_x})\\
&= \frac{1}{2}\frac{d}{dx}|\phi_x|^2-\partial_x(|\phi|^2)\frac{1}{4}|\phi|^4\\
&=\frac{d}{dx}\left(\frac{1}{2}|\phi_x|^2-\frac{1}{12}|\phi|^6\right).
\end{align*}
It implies that
\[
|\phi_x|^2-\frac{1}{6}|\phi|^6=0.
\] 
Hence, since $\phi_x$ is non vanishing, $\phi$ is also non vanishing on $\R$. We can write $\phi=\rho e^{i\theta}$ for $\rho>0$, $\rho,\theta \in \mathcal{C}^2(\R)$. Replacing $\phi=\rho e^{i\theta}$ in \eqref{eqstationary solution} we have
\begin{align*}
0&=(-\rho\theta_x^2+\rho_{xx}+\rho^3\theta_x)+i(2\rho_x\theta_x+\rho\theta_{xx}+\rho^2\rho_x).
\end{align*}
It implies that
\begin{align}
0 &=-\rho\theta_x^2+\rho_{xx}+\rho^3\theta_x .  \label{eq a}\end{align}
Replacing $\phi=\rho e^{i\theta}$ in \eqref{eq of phi} we have 
\begin{align*}
0&=\rho^2\theta_x+\frac{1}{4}\rho^4.
\end{align*}
Then $\theta_x=\frac{-1}{4}\rho^2$, replacing this equality in \eqref{eq a} we obtain 
\[
0=\rho_{xx}-\frac{5}{16}\rho^5.
\]
Multiplying the two sides of the above equality with $\rho_x$ we obtain
\begin{equation*}
0=\rho_{xx}\rho_x-\frac{5}{16}\rho^5\rho_x= \frac{d}{dx}\left(\frac{1}{2}\rho_x^2-\frac{5}{96}\rho^6\right).
\end{equation*} 
Hence,
\[
0=\rho_x^2-\frac{5}{48}\rho^6.
\]
Moreover, $\phi$ is non vanishing on $\R$ then $\rho>0$ and then $\rho_x$ is not change sign on $\R$. If $\rho_x>0$ then since $\rho(\infty)=0$ we have $\rho<0$ on $\R$, a contradiction. Hence, $\rho_x<0$ and $\rho_x=-\sqrt{\frac{5}{48}}\rho^3$. From this we easily check that
\[
\rho^2(x)=\frac{1}{\rho(0)^2+\sqrt{\frac{5}{12}}\, x},
\] 
which implies the contradiction, for the right hand side is not a continuous function on $\R$. This complete the proof.
\end{proof}

\begin{lemma}\label{lm 1}
The following is true:
\[
k-2\sqrt{B} \in L^2(\R), \quad k \in X^3(\R).
\]
\end{lemma}

\begin{proof}
Using $\phi \in L^{\infty}(\R)$ we obtain $k \in L^{\infty}(\R)$. On the other hand, since $\phi \in X^3(\R)$, we have $\phi_x \in L^2(\R)$, $\phi_{xx} \in L^2(\R)$ and it easy to see that
\begin{align*}
|\phi_x|^2 &=\frac{k_x^2}{4k}+k\theta_x^2 \in L^1(\R),\\
|\phi_{xx}|^2 &= \left|\frac{k_x\theta_x}{\sqrt{k}}+\theta_{xx}\sqrt{k}\right|^2+\left|\frac{k_{xx}}{2\sqrt{k}}-\sqrt{k}\theta_x^2-\frac{k_x^2}{4k\sqrt{k}}\right|^2\in L^1(\R).
\end{align*} 
This implies
\begin{align*}
\frac{k_x}{2\sqrt{k}} \in L^2(\R) & \text{ and }\sqrt{k}\theta_x \in L^2(\R) \\
\frac{k_x\theta_x}{\sqrt{k}}+\theta_{xx}\sqrt{k} \in L^2(\R) &\text{ and }\frac{k_{xx}}{2\sqrt{k}}-\sqrt{k}\theta_x^2-\frac{k_x^2}{4k\sqrt{k}} \in L^2(\R).
\end{align*}
Using $\sqrt{m}<k<\norm{k}_{L^{\infty}}$, $\theta_x=\frac{4B-k^2}{4k}\in L^{\infty}(\R)$, $k_x=2RR_x \in L^{\infty}$( \text{indeed }$|\phi_x|^2=|R_x|^2+|R\theta_x|^2 \in L^{\infty}(\R)$) we have
\begin{align*}
k_x\in L^2 &\text{ and } \theta_x \in L^2 ,\\
\theta_{xx} \in L^2 &\text{ and } k_{xx} \in L^2.
\end{align*}
By using $\theta_x=\frac{4B-k^2}{4k} \in L^2(\R)$, we have $4B-k^2 \in L^2(\R)$. Thus, $B \geq 0$ and $2\sqrt{B}-k \in L^2(\R)$. If $B=0$ then $k \in L^2(\R)$, hence, $R \in L^2(\R)$. It implies that $R \in H^1(\R)$, which contradicts the assumption of $\phi$. Thus, $B>0$. It remains to prove that $k_{xxx} \in L^2(\R)$. Indeed, from $\phi_{xxx} \in L^2(\R)$ we have
\begin{align}\label{equality of three derivative of phi fucntion}
|\phi_{xxx}|^2=|\theta_{xxx}\sqrt{k}+\mathcal{M}|^2+\left|\frac{k_{xxx}}{2\sqrt{k}}+\mathcal{N}\right|^2 \in L^1(\R)
\end{align}
where $\mathcal{M},\mathcal{N}$ are functions of $\theta,\theta_x,\theta_{xx},k,k_x,k_{xx}$. We can easily check that $\mathcal{M},\mathcal{N} \in L^2(\R)$. Hence, from \eqref{equality of three derivative of phi fucntion} and the facts that $\theta_x \in H^1(\R)$, $k \in X^2(\R)$, $k$ bounded from below we obtain $\theta_{xxx},k_{xxx} \in L^2(\R)$. This implies the desired results.
\end{proof}

From now on, we will denote $\phi_{B}$ is the stationary solution of \eqref{eqstationary solution} given by Theorem \ref{thm kink soliton} with $\theta_0=0$. We have the following asymptotic properties for $\phi_B$ at $\infty$. 
\begin{proposition}
Let $B>0$ and $\phi_B$ be kink solution of \eqref{eq:1}. Then for $x>0$, we have
\begin{align*}
|\phi_B-\sqrt{2\sqrt{B}}| &\lesssim e^{-\sqrt{B} x}.
\end{align*}
As consequence $\phi_B$ converges to $\sqrt{2\sqrt{B}}$ as $x$ tends to $\infty$ and there exists limit of $\phi_B$ as $x$ tends to $-\infty$.
\end{proposition}

\begin{proof}
We recall that
\begin{align}
\phi_B &=e^{i\theta}\sqrt{k}\; \label{eq property kink solution 1},\\
k(x)&=2\sqrt{B}+\frac{-1}{\sqrt{\frac{5}{72B}}\cosh(2\sqrt{B}x)+\frac{5}{12\sqrt{B}}} \; \label{eq property kink solution 2},\\
\theta(x) &=-\int_x^{\infty} \frac{B}{k(y)}-\frac{k(y)}{4} \; \label{eq property kink solution 3}\, dy.
\end{align}

From \eqref{eq property kink solution 2} we have 
\[
|k-2\sqrt{B}| \lesssim e^{-2\sqrt{B}x}. 
\]
Hence, for all $x\in \R$ we have
\begin{align}
|\phi_B(x)-\sqrt{2\sqrt{B}}| &\lesssim |e^{i\theta(x)}\sqrt{k(x)}-\sqrt{k(x)}|+|\sqrt{k(x)}-\sqrt{2\sqrt{B}}| \\
&\lesssim \norm{k}^{\frac{1}{2}}_{L^{\infty}}|e^{i\theta(x)}-1|+e^{-\sqrt{B}x} \label{eq lemma kink}
\end{align}
Moreover, for $x>0$, we have
\begin{align*}
|e^{i\theta(x)}-1| &\leq |\theta(x)| \leq\int_x^{\infty}\left|\frac{B}{k}-\frac{k}{4}\right|\,dx \\
&\leq \int_x^{\infty}\left|\frac{B}{k}-\frac{\sqrt{B}}{2}\right|+\left|\frac{\sqrt{B}}{2}-\frac{k}{4}\right|\,dx\\
&\lesssim \int_x^{\infty}\left|k-2\sqrt{B}\right|\,dx \lesssim \int_x^{\infty}e^{-2\sqrt{B}x}\,dx \lesssim e^{-2\sqrt{B}x}.
\end{align*}
Combining with \eqref{eq lemma kink} we obtain 
\[
|\phi_B(x)-\sqrt{2\sqrt{B}}| \lesssim e^{-\sqrt{B}x}.
\]
As consequence $\phi_B$ converges to $\sqrt{2\sqrt{B}}$ as $x$ tends to $\infty$. It is easy to check that $\theta$ converges to some constant when $x$ tends to $-\infty$, hence, there exists limit of $\phi_B$ when $x$ tends to $-\infty$. It is the desired result. 
\end{proof}

\bibliographystyle{abbrv}
\bibliography{bibliothequenonvanishingboundarycondition}

% \bibliographystyle{abbrv}
% \bibliography{biblio}

\end{document}